\documentclass[11pt,reqno,letterpaper]{amsart}
\usepackage{amssymb}
\usepackage{hyperref}
\usepackage{fullpage} 

\usepackage{cite} 

\usepackage{scalerel}[2014/03/10]
\usepackage[usestackEOL]{stackengine}

\def\dashint{\,\ThisStyle{\ensurestackMath{%
 \stackinset{c}{.2\LMpt}{c}{.5\LMpt}{\SavedStyle-}{\SavedStyle\phantom{\int}}}%
 \setbox0=\hbox{$\SavedStyle\int\,$}\kern-\wd0}\int}
\def\ddashint{\,\ThisStyle{\ensurestackMath{%
 \stackinset{c}{.2\LMpt}{c}{.5\LMpt+.2\LMex}{\SavedStyle-}{%
 \stackinset{c}{.2\LMpt}{c}{.5\LMpt-.2\LMex}{\SavedStyle-}{%
 \SavedStyle\phantom{\int}}}}\setbox0=\hbox{$\SavedStyle\int\,$}\kern-\wd0}\int}

\theoremstyle{plain}

\def \eps{\varepsilon}
\def \R{\mathbb R}
\def \lap{\Delta}

\begin{document}


\theoremstyle{plain}
\newtheorem{theorem}{Theorem} [section]
\newtheorem{corollary}[theorem]{Corollary}
\newtheorem{lemma}[theorem]{Lemma}
\newtheorem{proposition}[theorem]{Proposition}
\newtheorem{example}[theorem]{Example}


\theoremstyle{definition}
\newtheorem{definition}[theorem]{Definition}
\theoremstyle{remark}
\newtheorem{remark}[theorem]{Remark}

\numberwithin{theorem}{section}
\numberwithin{equation}{section}

\title{Pointwise mean-value formulas with quantitative remainder for higher-order Poisson equations}

\author[F. Charro]{Fernando Charro}
\address{Department of Mathematics, Wayne State University, 656 W. Kirby, Detroit, MI 48202, USA}
\email{fcharro@wayne.edu}
\thanks{}

\author[S. Hertrich]{Sophia Hertrich}
\address{Department of Mathematics, Wayne State University, 656 W. Kirby, Detroit, MI 48202, USA}
\email{sophia.hertrich@wayne.edu}

\author[C. Lebiedzik]{Catherine Lebiedzik}
\address{Department of Mathematics, Wayne State University, 656 W. Kirby, Detroit, MI 48202, USA}
\email{catherine.lebiedzik@wayne.edu}

\date{}

\begin{abstract}
Higher-order Poisson equations involve an integer power of the Laplacian and a nonzero forcing term, and appear in areas of physics and engineering such as hydrodynamics, structural engineering, and image processing. 
We introduce a family of mean-value formulas for solutions to higher-order Poisson equations, given in terms of linear combinations of iterated means, with an exact remainder quantified by the oscillation of the forcing term.
We also prove a regularity result and a strong converse to the mean-value property, in the sense that merely locally integrable functions
satisfying our formulas are regular solutions of the higher-order
Poisson equation.
In contrast with the homogeneous case, where the mean-value property forces smoothness, the regularity attainable here is dictated by the regularity of the forcing term.
Together, our results provide a mean-value characterization of solutions to higher-order Poisson equations.
\end{abstract}

\maketitle

\section{Introduction}
It is a well-established fact that harmonic functions can be characterized through the mean-value property: a function is harmonic if and only if at all points its value matches its average over all small scales. That is, $\Delta u (x) =0$ in $\Omega\subset\mathbb{R}^n$ if and only if
$$
 u(x) = \dashint_{B_\varepsilon(x)} u(y)\,dy \ \ \mbox{for every } B_\varepsilon (x) \subset \Omega.
$$ 
The asymptotic mean-value property, see \cite{Blaschke,Ku,Privaloff}, allows an infinitesimal detachment between  $u$ and its average of order strictly smaller than $\varepsilon^2$ as $\varepsilon\to0$.
 For a regular $u$, a Taylor expansion~shows 
\[
\Delta u(x)=\lim_{\varepsilon \to0}\frac{2(n+2)}{\varepsilon^2}\left(\dashint_{B_\varepsilon(x)} u(y)\,dy-u(x)\right).
\]
Hence, for $f\in C(\Omega)$, a real function $u$ satisfies the Poisson equation $\Delta u=f$ in $\Omega$ if and only if 
\begin{equation}
\label{Laplaciano.formula.asintotica}
u(x)=\dashint_{B_\varepsilon (x)} u(y)\,dy-\frac{\varepsilon^2}{2(n+2)}\,f(x)+o(\varepsilon^2)\qquad\textrm{as} \ \varepsilon \to0
\end{equation}
 for each $x \in \Omega$.
 As usual, we use $o(\varepsilon^2)$ to denote a quantity 
such that $o(\varepsilon^2)/\varepsilon^2 \to 0 $
as $\varepsilon \to 0$.

We say that a function $u$ is $m$-harmonic in $\Omega$ (or polyharmonic of order $m$) if and only if $u\in C^{2m}(\Omega)$ and $\Delta^{m} u=0$ in $\Omega$, where $\Delta^{2}u=\Delta(\Delta u)$, $\Delta^{3}u=\Delta(\Delta^2 u)$ and so forth (see \cite{Aronszajn.et.al.1983,Gazzola.et.al.2010} for background on polyharmonic functions and operators).
Polyharmonic problems appear in linear elasticity, low-Reynolds-number hydrodynamics, and structural engineering (biharmonic case) \cite{Meleshko.2003}, in modeling ulcers and viscous fluids (triharmonic) \cite{Ugail.Wilson.2005,Lesnic.2009}, 
in the theory of octonionic functions and linear-plane strain of decagonal quasicrystals \cite{Burdik.et.al.2019,Li.et.al.1999} (quadri-harmonic), and in
shape optimization (quinti-harmonic) \cite{Kirmani.Jamil.2018}. 
 Many popular digital-image compression and reconstruction algorithms use polyharmonic functions (especially bi- and triharmonic) \cite{Barbu.2020,Chui.2009,Damelin.Hoang.2018,Kubiesa.et.al.2004,Ugail.2004}.
Polyharmonic splines, or thin plate splines, are also used in data interpolation, computer graphics, and image processing \cite{Madych.Nelson.1990}.

 From a practical standpoint, these models become all the more difficult to manage the higher the order of the equation (for instance, the quinti-harmonic equation is of tenth order). 
Generally speaking, mean-value formulas, known to characterize many linear and nonlinear PDEs (see \cite{[Blanc et al. 2021.other.ops]} and the references therein), hold under more lenient regularity conditions than the corresponding PDEs, which could lead to new algorithms by replacing the PDEs in these applications.

We find several different mean-value characterizations of polyharmonicity in the literature.
A celebrated formula due to Pizzetti, see \cite{Pizzetti.1909, O}, expresses the average of an arbitrary $C^{2m}$ function over a ball or sphere as a power series involving its iterated Laplacians. 
Namely, given $u\in C^{2m}(\Omega)$,
\begin{equation}\label{Pizzetti.formula.volume}
\dashint_{B_{\varepsilon}(x)}u(y)\,dy
=
u(x)+\sum_{k=1}^{m}c_{k}\,\Delta^{k}u(x)\,\varepsilon^{2k}+o(\varepsilon^{2m})
\quad\textrm{as} \ \varepsilon \to0,
\end{equation}
 for 
$c_k=1/\big(
2^{k} k! \prod_{j=1}^{k}(n+2j)\big)$. 
There is a corresponding formula for the spherical mean, which only differs from \eqref{Pizzetti.formula.volume} in the coefficients in the expansion, given by $d_k= c_k\, (n+2k)/n$.
Sbrana  proved the converse result in dimension $n=2$, see \cite{Sbr}, specifically, that if a function with finite, integrable derivatives of order $2m-2$ in a domain $\Omega$ satisfies Pizzetti's mean-value formula, then it is $m$-harmonic in $\Omega$.
Given that the Pizzetti-Sbrana mean-value property involved derivatives of the solution, subsequent research sought mean-value characterizations of polyharmonic functions that did not involve derivatives.
 In \cite{NM, Nicolesco.book.1936}, Nicolesco 
 proved  that 
 $u\in L_{loc}^{1}(\Omega)$ is $m$-harmonic in $\Omega$ if and only if for a.e.\ $x\in \Omega$ and a.e.\ $\varepsilon$ with $0<\varepsilon<\textnormal{dist}(x, \partial \Omega)$, it holds
\[
 u(x)
 \cdot
 \begin{vmatrix}
 1 & 1 & \ldots & 1\\
 1 & \frac{n}{n+2} & \ldots & \frac{n}{n+2(m-1)}\\
 \vdots & \vdots & \ddots & \vdots\\
 1 & \frac{n^{m-1}}{(n+2)^{m-1}} & \ldots & \frac{n^{m-1}}{(n+2m-2)^{m-1}}
 \end{vmatrix}
=
 \begin{vmatrix}
 \mu_{0}(u,x,\varepsilon) & 1 & \ldots & 1\\
 \mu_{1}(u,x,\varepsilon) & \frac{n}{n+2} & \ldots & \frac{n}{n+2(m-1)}\\
 \vdots & \vdots & \ddots & \vdots\\
 \mu_{m-1}(u,x,\varepsilon)& \frac{n^{m-1}}{(n+2)^{m-1}} & \ldots & \frac{n^{m-1}}{(n+2m-2)^{m-1}}
 \end{vmatrix},
\]
for the iterated means 
\[
\mu_{k}(u,x,\varepsilon)
=
\begin{cases}
\displaystyle\dashint_{\partial B_{\varepsilon}(x)}u\,d\mathcal{H}^{n-1}(y)&\textrm{for}\ k=0\vspace{5pt}\\
\displaystyle
\dashint_{ B_{\varepsilon}(x)}u\,dy=
\frac{n}{\varepsilon^{n}}\int_{0}^{\varepsilon} t^{n-1}\mu_{0}(u,t,x)\,dt
&\textrm{for}\ k=1\vspace{5pt}\\
\displaystyle\frac{n}{\varepsilon^{n}}\int_{0}^{\varepsilon} t^{n-1}\mu_{k-1}(u,t,x)\,dt&\textrm{for}\ k\geq2.
\end{cases}
\]
Alternatively, \cite{M,Fichera,Caramanica.1987,BHP,Caramuta.Cialdea.2014}
 proved different mean-value properties for $m$-harmonic functions, in which the mean-values are considered on concentric spheres. 
Further mean-value theorems for polyharmonic functions have been obtained in \cite{CM,Lysik.2011,Lysik.2012,Lysik.2015,Lysik.2016,LZ}, and others.

In \cite{CLR}, we introduced a new family of exact and asymptotic mean-value formulas (see Theorem \ref{thm.MVP.polyharmonic.intro} below) given by linear combinations of iterated ball averages, i.e.,
\begin{equation*}
A_\varepsilon^k[v](x) = A_\varepsilon\big[A_\varepsilon^{k-1}[v]\big](x)
\quad\textrm{for}\quad
A_\varepsilon[v](x) = \dashint_{B_\varepsilon(x)} v(y) dy
\quad\textrm{and}\quad
\ k=2,3\ldots
\end{equation*}
and similarly for the spherical averages 
\[
\mathcal{A}_\varepsilon^k[v](x) = \mathcal{A}_\varepsilon\big[\mathcal{A}_\varepsilon^{k-1}[v]\big](x)
\quad\textrm{for}\quad
\mathcal{A}_{\varepsilon}[v](x) = \dashint_{\partial B_{\varepsilon}(x)} v(y)\,d\mathcal{H}^{n-1}(y).
\]
Notice that for an iterated average of either kind to be well-defined, for every $x$, we must choose $\varepsilon$ small enough such that $\varepsilon<\textnormal{dist}(x,\partial\Omega)/k$.
With this notation, one can write the classical mean-value formula for harmonic functions as
$u(x)=A_\varepsilon[u](x)=\mathcal{A}_\varepsilon[u](x).$ 
Linear combinations of iterated averages can also be seen as a real polynomial on the averaging operator $A_\varepsilon$. From this perspective, 
the characterization in \cite{CLR} is driven by a simple algebraic condition. In order for a mean-value formula of the form
\[
 P(A_{\varepsilon})[u](x)= \sum_{k=0}^{m}a_{k}\,A_{\varepsilon}^k[u](x)=0
\]
to characterize polyharmonicity, $t=1$ must be a root of the polynomial $P(t)$, and its multiplicity as a root determines the degree of polyharmonicity of the function.

There are also asymptotic versions of these mean-value properties, where the property holds up to a `little-$o$' asymptotic error.
However, as pointed out in \cite{CLR}, when considering iterated means one has to require locally uniform limits.
\begin{definition}[{\rm $o(r^k)$ locally uniformly}] \label{loc_unif} We say
$f(x,r)=g(x,r)+o({r}^{k})$ as ${r} \to 0$ \emph{locally uniformly in $\Omega$} if for every
compact subset~$K\subset \Omega$, given $\eta>0$, there exists $r_0$ such that
\[
\operatorname*{ess\,sup}_{x\in K}|f(x,r)-g(x,r)|\leq\eta\, {r}^k\quad\textrm{for all}\ 0<r<r_0.
\]
\end{definition}

In its simplest case, the characterization of polyharmonic functions introduced in \cite{CLR} is as follows.
\begin{theorem}[Mean-value characterization of polyharmonic functions, \cite{CLR}]\label{thm.MVP.polyharmonic.intro}
 Let $\Omega\subset\mathbb{R}^n
$ be an open set, $m$ a positive integer, and $u:\Omega\to\R$ with $u \in L^1_\textnormal{loc}(\Omega)$.
 Then, the following are equivalent:
\begin{enumerate}\itemsep3pt
 \item $u$ is $m$-harmonic in $\Omega$, i.e., $u\in C^{2m}(\Omega)$ and $\Delta^mu=0$ in $\Omega$ in the classical sense.
 
\item For almost every $x\in\Omega$ and every ${\varepsilon}<\textnormal{dist}(x,\partial\Omega)/m$, it holds that
\begin{equation}\label{MVP.mharmonic.main.intro}
 (A_{\varepsilon}-I)^m[u](x)=0,
 \end{equation} 
 or, equivalently,
\begin{equation*}
 u(x)=\sum_{j=1}^{m}\binom{m}{j}(-1)^{j-1}A_{\varepsilon}^j[u](x).
 \end{equation*}
\item It holds that \begin{equation*}
 (A_{\eps}-I)^m[u](x)=o(\eps^{2m}),
\quad\textrm{as $\eps\to0$ locally uniformly in $\Omega$.}
 \end{equation*}

\end{enumerate}
The above equivalences also hold with the spherical average $\mathcal{A}_\eps$ in place of $A_{\varepsilon}$.
\end{theorem}

Unlike previous mean-value formulas in the literature, \eqref{MVP.mharmonic.main.intro} has a nice geometric interpretation consistent with the classical harmonic formula. For a harmonic function, the value at a point must coincide with the mean-value at neighboring points, while for biharmonic functions, the difference between the function and its average must match the difference between the first and second averages. 
For general $m$-harmonic functions, the value of $(A_\varepsilon-I)^{m-1}[u]$ and its average must coincide for all $\varepsilon$ small enough.

These mean-value formulas have potential applications in areas where polyharmonic equations arise (such as elasticity, plate theory, and image processing, mentioned above) as well as in non-Euclidean settings where iterated averaging is more natural than differential operators.
In fact, formulas such as \eqref{MVP.mharmonic.main.intro}
can provide a unified approach to polyharmonicity in diverse settings. For instance, it can provide a definition of $\Delta^m$ in non-Euclidean contexts such as Riemannian manifolds and Carnot groups, where one can make sense of $A_r$ for locally integrable functions.
Formula \eqref{MVP.mharmonic.main.intro} can also yield new ways to define polyharmonic functions in graphs by replacing $A_\varepsilon$ with a discrete average, which could lead to the development of numerical methods.

Interestingly, certain linear combinations of iterated averages of a merely locally integrable function force the function to be smooth. As shown in \cite{CLR}, seeing the linear combination as a polynomial in the averaging operator,  the corresponding mean-value property is smoothing when the polynomial has 1 as a root. In particular, we have the~following.

\begin{theorem}[Smoothness from mean-value formulas, \cite{CLR}]\label{thm.regularity}
 Let $\Omega\subset\mathbb{R}^n
$ be an open set, $m$ a positive integer, and $u:\Omega\to\R$.
 Assume that $u\in L^1_\textnormal{loc} (\Omega)$ satisfies the asymptotic mean-value property \begin{equation}\label{AMVP.regularity.intro}
 (A_{\eps}-I)^m[u](x)=o(\eps^{2m}) \quad
 \textrm{as}\ \eps\to0\ \textrm{locally uniformly in}\ \Omega.
 \end{equation}
Then $u\in\ C^\infty(\Omega)$.
 In particular, the same conclusion is true
when we replace \eqref{AMVP.regularity.intro} by
\[
 (A_{\eps}-I)^m[u](x)=0 \quad
 \textrm{for a.e.\ $x\in\Omega$ and every $\varepsilon<\frac{\textnormal{dist}(x,\partial\Omega)}{m}$.}
\]
The above results also hold with $\mathcal{A}_{\eps}$ in place of $A_{\varepsilon}$.

\end{theorem}

Based on the above results and the rich applications of the homogeneous polyharmonic equation, a natural next question is to consider nonzero right-hand sides and develop mean-value properties for the higher-order Poisson equation $\Delta^mu=f$.
There are two main ways one can go about this generalization, by considering asymptotic mean-value formulas, where the formula holds up to a small error, or by considering exact formulas with a quantitative expression of the remainder. We pursue the latter in this manuscript.

The rest of the paper is structured as follows.
In Section \ref{main.results}, we state our main results: mean-value properties with exact remainder given in terms of the oscillation of the right-hand side for the classical and higher-order Poisson equations. These are found in Theorems~\ref{MVP.pointwise.Laplacian} and \ref{general.main}, and Corollaries \ref{cor.spherical.m=1} and \ref{cor.spherical.higher}.
We also provide a discussion of the results, compare asymptotic mean-value properties versus our exact mean-value properties with quantitative remainder, and prove a strong converse to the mean-value property.
In Section \ref{section.m.is.1}, 
we prove
Theorem \ref{MVP.pointwise.Laplacian} and
 Corollary \ref{cor.spherical.m=1}, while in Section \ref{section.m.greater.1}, we prove
Theorem  \ref{general.main}  and Corollary \ref{cor.spherical.higher}.

\section{Main results}\label{main.results}

We first show a mean-value property with exact remainder for the classical Poisson problem. We provide
two formulations of the mean-value property, given in \eqref{with_A} and \eqref{MVP_lap} below. Formula \eqref{with_A} has an integral right-hand side and holds pointwise. 
On the other hand, formula \eqref{MVP_lap} lines up with the
asymptotic mean-value property \eqref{Laplaciano.formula.asintotica}, while providing an explicit remainder in terms of the oscillation of $f$. Although no continuity of
$f$ is needed for the identity to hold, any modulus of continuity of
$f$ translates directly into an improvement of the order of the remainder.

\subsection{Exact mean-value formulas with quantitative remainder.}

Before stating our results, let us define the radial kernel \begin{equation}\label{K_def}
K_{\varepsilon}(r) 
= \frac{r}{n}\bigg( 1 - \left(\frac{r}{\eps}\right)^n\bigg)\qquad\textrm{for}\ 0\leq r\leq \eps,
\end{equation}
where we assume, here and in the sequel that $n\geq2$.
Observe that $K_{\varepsilon}(r)$ is zero both when $r=0$ and $r=\varepsilon$ and positive in between.
A key property of this kernel, central to our results below, is that its moments provide an integral representation of the Pizzetti coefficients in
\eqref{Pizzetti.formula.volume};
see Lemma~\ref{Kint} below. The simplest instance is the integral of the~kernel,
\[
\int_0^\eps K_\eps(r)\,dr
=\eps^{2}c_1=\frac{\eps^{2}}{2(n+2)},
\]
a constant the reader may recognize from 
\eqref{Laplaciano.formula.asintotica}.

\begin{theorem}\label{MVP.pointwise.Laplacian} 
Let $\Omega\subset\R^n$, $n\geq2,$ be open, $v, f :\Omega\to\R$ and $f\in L^\infty_{loc}(\Omega)$. Let  $K_{\varepsilon}$ be given by \eqref{K_def}.
Assume $v\in C^{1,1}(\Omega)$ satisfies
$\Delta v = f$ a.e.~in~$\Omega$. Then, the following mean-value property~holds 
\begin{equation}\label{with_A}
(A_\eps-I)[v](x)
=
\int_0^\eps A_r[f](x)\, K_{\varepsilon}(r)\,dr
\end{equation}
for all $x\in\Omega$ and every ${\varepsilon}<\textnormal{dist}(x,\partial\Omega)$. 
Alternatively, one can write
\begin{equation}\label{MVP_lap}
(A_\eps-I)[v](x)=\frac{f(x)}{2(n+2)}\varepsilon^2
+R_\varepsilon(x),
\end{equation}
for almost every $x\in\Omega$ and every ${\varepsilon}<\textnormal{dist}(x,\partial\Omega)$,
where 
\begin{equation}\label{R_MVP_lap}
R_\varepsilon(x)=
\int_{0}^{\varepsilon}
A_r \big[f(\cdot)-f(x)\big](x)
\,
K_{\varepsilon}(r)\,dr
=
\int_{0}^{\varepsilon}
\dashint_{B_r(x)} \big(f(y)-f(x)\big)\,dy
\,
K_{\varepsilon}(r)\,dr.
\end{equation}
In particular:
\begin{enumerate}
 \item If $f$ is continuous at $x$, then $R_\varepsilon(x)=o(\varepsilon^2)$ as $\varepsilon\to0$.
If $f \in C(\Omega)$, then \eqref{MVP_lap} holds pointwise in $\Omega$
 and $R_\varepsilon = o(\varepsilon^2)$
 locally uniformly in $\Omega.$

 \item If $f$ has modulus of continuity $\omega$, i.e.,
$$|f(x)-f(y)|\le\omega(|x-y|)\quad\textrm{for all $x,y\in\Omega$,}
$$
where $\omega:[0,\infty)\to[0,\infty)$ is non-decreasing with
$\lim_{t\to0^+}\omega(t)=0$, then
$$|R_\varepsilon(x)|\leq \frac{1}{2(n+2)}\,\omega(\varepsilon)\,\varepsilon^2\quad\textrm{for every}\ \varepsilon<\textnormal{dist}(x,\partial\Omega).$$
 
 \item If $f$ is H\"older or Lipschitz continuous with exponent $\alpha\in(0,1]$, then $R_\varepsilon=O(\varepsilon^{2+\alpha})$.
\end{enumerate}
\end{theorem}

\begin{remark}
In the notation $f(\cdot)-f(x)$ in \eqref{R_MVP_lap}, $f(x)$ denotes the value of the function at the center of the ball $B_r(x),$ and is constant relative to the average. The average takes place over the integration variable $y$, which is suppressed in the $A_r$ notation.
\end{remark}

\begin{remark}
The hypothesis $f \in L^\infty_{loc}(\Omega)$ can be seen as a compatibility condition, since, by Rademacher's theorem (applied to $\nabla v$), the second derivatives of $v$ exist a.e.\ and
are locally bounded, which implies
$\Delta v \in L^\infty_{loc}(\Omega)$.
Therefore, \eqref{with_A} holds at all points $x\in\Omega$ because the right-hand side sees $f$ only through integrals, and is defined
and finite pointwise.
On the other hand,
\eqref{MVP_lap}, 
which evaluates $f\in L^\infty_{loc}(\Omega)$ at $x$,
holds only almost everywhere.
\end{remark}

We can derive spherical mean-value formulas from the ball-average ones in Theorem \ref{MVP.pointwise.Laplacian}. \begin{corollary}\label{cor.spherical.m=1}
Let $\Omega$, $f$, and $v$ be as in Theorem \ref{MVP.pointwise.Laplacian}.
 Then, for every $x\in\Omega$ and every
$0<\eps<\textnormal{dist}(x,\partial\Omega)$,
\begin{equation}\label{MVP.spherical}
(\mathcal{A}_{\eps}-I)[v](x)
=\int_0^{\eps} A_r[f](x)\,\frac{r}{n}\,dr
=\int_0^{\eps}\mathcal{A}_r[f](x)\,\mathcal{K}_\eps(r)\,dr,
\end{equation}
for the kernel
\vspace{-3pt}
\begin{equation}\label{ K.spherical.def}
\vspace{+3pt}
\mathcal{K}_\eps(r)=
\begin{cases}
\displaystyle \frac{r}{n-2}\bigg(1-\Big(\frac{r}{\eps}\Big)^{n-2}\bigg)
&\textrm{if }n\geq3,\\[10pt]
\displaystyle r\log\Big(\frac{\eps}{r}\Big)&\textrm{if }n=2.
\end{cases}
\end{equation}
Moreover, for a.e.\ $x\in\Omega$ and every
$\eps<\textnormal{dist}(x,\partial\Omega)$,
\begin{equation}\label{MVP_lap.spherical}
(\mathcal{A}_{\eps}-I)[v](x)
=
\frac{f(x)}{2n}\,\eps^2+\mathcal{R}_\eps(x),
\end{equation}
where
\begin{equation}\label{R_MVP_lap.spherical}
\mathcal{R}_\eps(x)
=\int_0^{\eps} A_r\big[f(\cdot)-f(x)\big](x)\,\frac{r}{n}\,dr
=\int_0^{\eps}\mathcal{A}_r\big[f(\cdot)-f(x)\big](x)\,
\mathcal{K}_{\eps}(r)\,dr.
\end{equation}
The analogues of statements (1)--(3) in
Theorem \ref{MVP.pointwise.Laplacian} hold 
with the constant $\frac{1}{2(n+2)}$ replaced by $\frac{1}{2n}$.
\end{corollary}

\begin{remark}\label{remark.K.mathcal.K.dimension.intro}
When $n\geq3$ we can obtain $\mathcal{K}_\eps$ from
$K_\eps$ upon replacing $n$ by $n-2$; the case $n=2$ arises as the
limit as $n\to2^+$. 
\end{remark}

Our main result extends Theorem \ref{MVP.pointwise.Laplacian} to the higher-order Poisson equation $\Delta^mu=f$, for $m$ a positive integer. 
Recall that Theorem \ref{MVP.pointwise.Laplacian} provides two forms of the mean-value property, \eqref{with_A}, in which the right-hand side is a weighted integral of ball averages of $f$, and \eqref{MVP_lap}, in which the right-hand side features the pointwise value 
$f(x)$ plus a remainder given by a weighted integral in $r$ of the averaged increments $(A_r - I)[f](x)$.
For $m>1$, the integral form \eqref{with_A} extends to an $m$-dimensional weighted integral of $m$ iterated ball averages of $f$.
The extension of the pointwise form \eqref{MVP_lap}, by contrast, involves a combination of all averaged increments $(A_\eps-I)^j[f](x)$ with $0\leq j\leq m-1$ (instead of just $f(x)$ and $(A_r-I)[f](x)$),
along with coefficients assembled from the Pizzetti constants.

We now introduce the combinatorial coefficients that organize the
higher-order remainder expansion; the interested reader is referred to
\cite[Section 1.2]{Stanley.2012}.
For integers $1\le k\le m$, let $\mathcal{C}(m,k)$ denote the set of
\emph{compositions} of $m$ into $k$ positive parts, i.e., ordered tuples
$\alpha=(\alpha_1,\dots,\alpha_k)\in\mathbb{N}^k$ with $\alpha_i\ge1$ and
$\alpha_1+\cdots+\alpha_k=m$.
Given $\alpha\in\mathcal{C}(m,k)$, we write $\sigma_i=\alpha_1+\cdots+\alpha_i$ for the
partial sums (with $\sigma_0=0$ by convention, and $\sigma_k=m$). Then, the intervals of
integers
\[
I_i=\{\sigma_{i-1}+1,\dots,\sigma_i\},\qquad 1\le i\le k,
\]
partition $\{1,\dots,m\}$ into $k$ consecutive blocks, the $i$-th block
containing $\alpha_i$ elements. Within the $i$-th block, the $j$-th
element is $\sigma_{i-1}+j$. In particular,
\begin{equation}\label{block.bijection}
\begin{split}
(i,j)\longmapsto \sigma_{i-1}+j
\quad&\textrm{is a bijection from}\\
&\big\{(i,j):\,1\le i\le k,\ 1\le j\le\alpha_i\big\}
\quad\textrm{onto}\quad \{1,\dots,m\},
\end{split}
\end{equation}
since, for each $1\le i\le k$ and $1\le j\le\alpha_i$, we have
$\sigma_{i-1}+j\in I_i\subset\{1,\dots,m\}$, and, conversely, each
$q\in\{1,\dots,m\}$ lies in exactly one block $I_i$, where it occupies exactly
one position $j=q-\sigma_{i-1}$.

For every $1\le k\le m$ and $r=(r_1,\dots,r_m)\in[0,\eps]^m$, we define
\begin{equation}\label{B_def}
b_{\eps}^{m,k}(r)
=\sum_{\alpha\in\mathcal{C}(m,k)}\;
\prod_{i=1}^{k}\prod_{j=1}^{\alpha_i}
\left(\frac{r_{\sigma_{i-1}+j}}{\eps}\right)^{2j-2}.
\end{equation}
Note that
$b^{m,m}_\eps\equiv1$ (a single composition with all blocks of length one, so every factor is 1), while $b^{m,1}_\eps(r)=\prod_{q=1}^m\left(r_q/\eps\right)^{2q-2}$
(a single block). 
Classifying the compositions in \eqref{B_def} according to the length
$p\in\{1,\dots,m-k+1\}$ of their last block yields the recursion
\[
b^{m,k}_\eps(r_1,\dots,r_m)
=\sum_{p=1}^{m-k+1}
b^{m-p,\,k-1}_\eps(r_1,\dots,r_{m-p})
\prod_{j=1}^{p}\left(\frac{r_{m-p+j}}{\eps}\right)^{2j-2},
\qquad b^{m,m}_\eps\equiv1,
\]
convenient for generating low-order coefficients; e.g., from
$b^{1,1}_\eps\equiv1$ and $b^{2,1}_\eps(r)=(r_2/\eps)^2$, it produces
$b^{3,2}_\eps(r)=\left(r_2/\eps\right)^{2}+\left(r_3/\eps\right)^{2}$.
We will discuss further properties of the coefficients $b^{m,k}_{\eps}$
in Lemma~\ref{lemma.B.properties}~below.

We are ready to state our main result. To simplify the notation, and emphasize the parallel with Theorem \ref{MVP.pointwise.Laplacian}, we will use multi-index notation. In particular, we will denote $dr=dr_1dr_2\ldots dr_m$ and $A_{r}[f]=\big(A_{r_1}\!\cdots A_{r_{m}}\big)[f]$ for the composition of averages.
We have the following.

\begin{theorem}\label{general.main}
Let $\Omega\subset\R^n$, $n\geq2,$ be open, $u, f:\Omega\to\R$ and $f\in L^\infty_{loc}(\Omega)$.
Assume $u\in C^{2m-1,1}(\Omega)$ satisfies $\lap^m u=f$ a.e.\ in $\Omega$. Then the following mean-value property holds for 
every $x\in\Omega$ and every ${\varepsilon}<\textnormal{dist}(x,\partial\Omega)/m$:
\begin{equation}\label{MVP_1}
(A_\eps-I)^{m}[u](x) 
= \int_{[0,\eps]^m} A_{r}[f](x)\cdot\prod_{i=1}^{m}K_\eps(r_i)\,dr,
\end{equation}
for $K_\eps$ given by \eqref{K_def}.
Equivalently, for 
almost every $x\in\Omega$ and every ${\varepsilon}<\textnormal{dist}(x,\partial\Omega)/m$, it holds
\begin{equation}\label{MVP_2}
(A_\eps-I)^m[u](x)=
\varepsilon^{2m}\sum_{l=0}^{m-1} \tilde{c}_{l}\,(A_\eps -I)^l[f](x) 
+R_\varepsilon(x),
\end{equation}
for the constants $\tilde{c}_l$ defined as follows in terms of the Pizzetti coefficients
\begin{equation}\label{pizz_c}
\tilde{c}_l=\sum_{\alpha\in\mathcal{C}(m,m-l)}\;
\prod_{i=1}^{m-l} c_{\alpha_i},
\qquad\textrm{for}\quad
c_\alpha=\prod_{j=1}^\alpha\frac{1}{(2j)(n+2j)},
\end{equation}
and the remainder
\begin{equation}\label{R_MVP}
R_\varepsilon(x)=
\int_{[0,\eps]^m}
\bigg( A_{r}-\sum_{l=0}^{m-1} b_{\eps}^{m,m-l}(r)\,(A_\eps-I)^l\bigg)[f](x)\cdot
\prod_{i=1}^{m}K_\eps(r_i)\,dr,
\end{equation}
with the coefficients $b_{\eps}^{m,k}(r)$ given by \eqref{B_def}.
In particular:
\begin{enumerate}
 \item If $f$ is continuous at $x$, then $R_\varepsilon(x)=o(\varepsilon^{2m})$ 
 as $\varepsilon\to0$. If $f \in C(\Omega)$, then \eqref{MVP_2} holds pointwise 
 in $\Omega$ and $R_\varepsilon = o(\varepsilon^{2m})$ locally uniformly in $\Omega$.
 
 \item If $f$ has modulus of continuity $\omega$, i.e.,
 $$|f(x)-f(y)|\le\omega(|x-y|)\quad\textrm{for all $x,y\in\Omega$,}$$
 where $\omega:[0,\infty)\to[0,\infty)$ is non-decreasing with
 $\lim_{t\to0^+}\omega(t)=0$, then
 $$|R_\varepsilon(x)|
 \leq 
 \frac{3^{m-1}-2^{m-1}+1}{\big(2(n+2)\big)^m}\,\omega(m\varepsilon)\,\varepsilon^{2m}
 \quad\textrm{for every}\ \varepsilon<\textnormal{dist}(x,\partial\Omega)/m.$$
 
 \item If $f$ is H\"older or Lipschitz continuous with exponent 
 $\alpha\in(0,1]$, then $R_\varepsilon=O(\varepsilon^{2m+\alpha})$.
\end{enumerate}
\end{theorem}

Let us describe the heuristics of how the expansion \eqref{MVP_2} and coefficients \eqref{B_def} and \eqref{pizz_c} arise.
By Pizzetti's
formula \eqref{Pizzetti.formula.volume}, one can write the operator $A_\eps-I$ as a formal power series in the Laplacian as follows, 
\begin{equation}\label{formal.expansion}
A_\eps-I
=
\sum_{\alpha\ge1}c_\alpha\,(\eps^{2}\Delta)^{\alpha},
\end{equation}
for $c_\alpha$ the Pizzetti coefficients in \eqref{pizz_c}. 
From here, one can derive a formal expression for $(A_\eps-I)^{m}$ by taking
the $m$-th power on both sides of \eqref{formal.expansion}. Expanding the
right-hand side, collecting terms, and using the equation $\Delta^{m}u=f$,
one would expect to obtain a formula with the structure of \eqref{MVP_2}.
Compositions of $m$ arise at exactly this point.
 Indeed, splitting $(A_\eps-I)^{m}=(A_\eps-I)^{l}(A_\eps-I)^{m-l}$ for 
$0\le l\le m-1$, and
expanding only the last $m-l$ factors using \eqref{formal.expansion}, we see that each of
the factors contributes one term $c_{\alpha_i}(\eps^{2}\Delta)^{\alpha_i}$, leading to the
products $c_{\alpha_1}\cdots c_{\alpha_{m-l}} $, indexed by the ordered
tuples $(\alpha_1,\dots,\alpha_{m-l})$, which in turn lead to the coefficients $\tilde{c}_l$ in \eqref{pizz_c}.
We note that these coefficients are, in fact, ordinary partial Bell
polynomials evaluated at the Pizzetti coefficients,
$\tilde c_l=\hat B_{m,m-l}(c_1,c_2,\dots)$. This is to be expected,
since these polynomials are defined precisely as the coefficients
of an integer power of a power series with no constant term; see \cite[Sec.~3.3]{Comtet.1974}.
Now, to parallel the structure of Theorem
\ref{MVP.pointwise.Laplacian}, where the remainder \eqref{R_MVP_lap} is
obtained by subtracting the value $f(x)$ under the weighted integral
in \eqref{with_A}, we need, so to speak, a primitive of each
$\tilde c_l$ with respect to the kernel $\prod_{i=1}^m K_\eps(r_i)$, meaning a function of $r$ whose integral against the
kernel over $[0,\eps]^m$ returns $\eps^{2m}\,\tilde c_l$. The coefficients
$b^{m,m-l}_{\eps}(r)$ defined in \eqref{B_def} play precisely this role,
thanks to the integral representation of the
Pizzetti coefficients; see
Lemmas \ref{Kint} and \ref{lemma.B.properties} below.

We point out that in Theorem \ref{general.main}, one can consider different radii for each of the averaging operators in \eqref{MVP_1}.

\begin{proposition}\label{prop.multiradius}
Let $\Omega$, $f$, and $u$ be as in Theorem~\ref{general.main}. Then, for every
$x\in\Omega$ and every tuple of positive radii
$(\varepsilon_1,\dots,\varepsilon_m)$ with
$\varepsilon_1+\cdots+\varepsilon_m<\textnormal{dist}(x,\partial\Omega)$, we have
\begin{equation}\label{eq.multiradius} 
\prod_{i=1}^{m}\big(A_{\varepsilon_i}-I\big)[u](x)
=\int_0^{\varepsilon_1}\cdots\int_0^{\varepsilon_m}
\big(A_{r_1}\cdots A_{r_m}\big)[f](x)\,
\prod_{i=1}^{m}K_{\varepsilon_i}(r_i)\,dr_m\ldots dr_1.
\end{equation}
\end{proposition}

The proof of Proposition \ref{prop.multiradius} is the same as that of
\eqref{MVP_1} in Section \ref{section.m.greater.1}. The argument iterates the
case $m=1$ once per averaging operator and whether the radii are all equal
does not play a role; therefore, we will omit the details.
It is worth noting that, by contrast, \eqref{MVP_2} does not carry over canonically to this
setting. With equal radii, the correction terms
$(A_\varepsilon-I)^l[f](x)$ are unambiguous, whereas when the radii are different, one must decide which $l$ of the $m$ choices
$\varepsilon_1,\dots,\varepsilon_m$ appear in each of the correction
terms $\prod_{i=1}^l(A_{\varepsilon_{\alpha_i}}-I)[f](x)$ for
$1\leq l\leq m-1$.

We also have the spherical-average analogue of  Theorem~\ref{general.main}.

\begin{corollary}\label{cor.spherical.higher}
Let $\Omega$, $f$, and $u$ be as in 
Theorem~\ref{general.main}. Then, using multi-index notation,
\begin{equation}\label{MVP.spherical.higher}
(\mathcal{A}_{\eps}-I)^m[u](x)
=
\int_{[0,\eps]^m} A_{r}[f](x)\cdot
\prod_{i=1}^m\frac{r_i}{n}\,dr
=
\int_{[0,\eps]^m} \mathcal{A}_{r}[f](x)\cdot
\prod_{i=1}^m\mathcal{K}_{\eps}(r_i)\,dr,
\end{equation}
for every
$x\in\Omega$ and every
$0<\eps<\textnormal{dist}(x,\partial\Omega)/m$.
Moreover,
\begin{equation}\label{MVP_2.spherical}
(\mathcal{A}_{\eps}-I)^m[u](x)
=\eps^{2m}\sum_{l=0}^{m-1}\tilde d_l\,
(\mathcal{A}_{\eps}-I)^l[f](x)
+\mathcal{R}_{\eps}(x),
\end{equation}
 for a.e.\ $x\in\Omega$ and
a.e.\ $\eps$ with ${\varepsilon}<\textnormal{dist}(x,\partial\Omega)/m$,
where the constants $\tilde d_l$ are given in terms of the spherical Pizzetti coefficients as follows 
\begin{equation}\label{pizz_d}
\tilde d_l=\sum_{\alpha\in\mathcal{C}(m,m-l)}\;
\prod_{i=1}^{m-l} d_{\alpha_i},
\qquad
d_\alpha=\prod_{j=1}^{\alpha}\frac{1}{(2j)(n+2j-2)}
\end{equation}
and the remainder is
\begin{equation}\label{R_MVP.spherical}
\mathcal{R}_{\eps}(x)=
\int_{[0,\eps]^m}
\bigg(\mathcal{A}_{r}
-\sum_{l=0}^{m-1} b_{\eps}^{m,m-l}(r)\,(\mathcal{A}_\eps-I)^l\bigg)[f](x)\cdot
\prod_{i=1}^{m}\mathcal{K}_\eps(r_i)\,dr,
\end{equation}
with the \emph{same} coefficients $b^{m,k}_{\eps}$ as in \eqref{B_def},
which do not depend on the dimension.
The analogues of statements (1)--(3) in
Theorem \ref{general.main} hold with appropriate changes to the constants.
\end{corollary}

\begin{remark}
The restriction to almost every $\eps$ in \eqref{MVP_2.spherical} stems from the fact that $f\in L^\infty_{loc}(\Omega)$ is defined only up to a null set, and the
correction terms $(\mathcal{A}_\eps-I)^l[f](x)$ sample $f$ on spheres, which are null sets in~$\R^n$. 
By Fubini's theorem in polar coordinates, for each fixed $x$ the spherical mean $\mathcal{A}_\eps[f](x)$ is
well defined only for a.e.\ radius $\eps\in(0,\textnormal{dist}(x,\partial\Omega))$, and the same applies to the iterated means.
By contrast, the integral form \eqref{MVP.spherical.higher} only involves radial integrals of means of $f$, and therefore holds for every
$x\in\Omega$ and every admissible $\eps$. Whenever
$f\in C(\Omega)$, the spherical means of $f$ are defined for all radii and
\eqref{MVP_2.spherical} holds for every $x\in\Omega$ and every
$\eps<\textnormal{dist}(x,\partial\Omega)/m$.
\end{remark}

\subsection{Quantitative versus asymptotic mean-value properties}

When $f\in C(\Omega)$, 
Theorem \ref{general.main} yields the asymptotic mean-value property
\begin{equation}\label{AMVP.f}
(A_\eps-I)^m[u](x)
=
 \frac{f(x)}{\big(2(n+2)\big)^m}\,\varepsilon^{2m}+o(\eps^{2m})
\qquad\textrm{as}\ \eps\to0,\ \textrm{locally uniformly in}\ \Omega,
\end{equation}
since the remainder in \eqref{MVP_2} is
$o(\eps^{2m})$ locally uniformly by statement (1), and so are the correction terms with $l\geq1$ by Lemma \ref{lemma.osc} below and the uniform continuity of $f$
on compact subsets.
Let us recall from 
\cite{CLR} (see also Cheng \cite{CM}) that, already in the homogeneous case 
$f\equiv0$, the little-$o$ in asymptotic mean-value properties involving iterated 
means must be locally uniform in the sense of Definition~\ref{loc_unif}, as there are counterexamples otherwise. For instance, it is shown in \cite{CLR} that there are 
functions satisfying $(A_{r}-I)^2[u](x)=o(r^{4})$ pointwise in~$\Omega$ as $r\to0$ 
that are not biharmonic.

The following proposition shows that the continuity of $f$ is not only 
sufficient but also necessary for \eqref{AMVP.f} to hold.
In fact, we show that the locally uniform 
asymptotic mean-value property forces both $u$ and $f$ to be continuous 
(up to a null set). Consequently, for discontinuous 
$f\in L^\infty_{loc}(\Omega)$, no asymptotic formulation with locally uniform 
remainder is possible, and exact formulas such as \eqref{MVP_1}--\eqref{MVP_2} 
become the natural formulation of the mean-value property.

\begin{proposition}[Necessity of continuity in asymptotic mean-value 
properties]\label{prop.necessity.continuity}
Let $\Omega\subset\R^n$, $n\geq2$, be open, $m$ a positive integer, $f:\Omega\to\R,$
$f\in L^\infty_{loc}(\Omega)$. 
Assume that there exists a real-valued function $u\in L^1_{loc}(\Omega)$ satisfying \eqref{AMVP.f}. 
Then, both $u$ and $f$ agree a.e.\ in $\Omega$ with continuous~functions.
\end{proposition}

\begin{proof}
Let us write \eqref{AMVP.f} in the form
\begin{equation}\label{AMVP.f.remainder}
(A_\eps-I)^m[u](x)= \frac{f(x)}{\big(2(n+2)\big)^m}\,\varepsilon^{2m}+\rho_\eps(x),
\end{equation}
where, by Definition \ref{loc_unif}, for every compact $K\subset\Omega$ and 
every $\eta>0$ there exists $\eps_0>0$ such that
\begin{equation}\label{rho.smallness}
\operatorname*{ess\,sup}_{x\in K}|\rho_\eps(x)|\leq\eta\,\eps^{2m}
\qquad\textrm{for all}\ 0<\eps<\eps_0.
\end{equation}

Let us show first that $u$ agrees a.e.\ with a continuous function. 
Expanding the binomial in \eqref{AMVP.f.remainder} and isolating the zero-order term, we can write, for 
a.e.\ $x\in\Omega$ and $\eps<\textnormal{dist}(x,\partial\Omega)/m$,
\[
\frac{f(x)}{\big(2(n+2)\big)^m}\,\varepsilon^{2m}+\rho_\eps(x)
=
(A_\eps-I)^m[u](x)
=(-1)^m u(x)+\sum_{j=1}^{m}\binom{m}{j}(-1)^{m-j}A_\eps^{j}[u](x).
\]
Denoting 
\[
h_\eps=\sum_{j=1}^{m}\binom{m}{j}(-1)^{j-1}A_\eps^{j}[u],
\]
and rearranging terms, we get
\begin{equation}\label{u.as.uniform.limit}
u(x)
=
h_\eps(x)
+(-1)^m\left(\frac{f(x)}{\big(2(n+2)\big)^m}\,\eps^{2m}+\rho_\eps(x)\right).
\end{equation}
Let $K\subset\Omega$ be a closed ball and 
$\eps<\textnormal{dist}(K,\partial\Omega)/m$. 
By the Lebesgue lemma,
 $x \mapsto A_\eps^ju$ is continuous in ${\{x\in\Omega:\ \textnormal{dist}(x,\partial\Omega)>m\eps\}}$, and therefore $h_\eps$ is continuous on $K$. 
Then, from \eqref{u.as.uniform.limit} and \eqref{rho.smallness} with $\eta=1$, we get 
\[
\operatorname*{ess\,sup}_{K}|u-h_\eps|
\leq
\left(\frac{\|f\|_{L^\infty(K)}}{\big(2(n+2)\big)^m}+1\right)\eps^{2m}
\]
for all $\eps$ small enough. Choosing a sequence $\{\eps_j\}$ that makes
${\operatorname*{ess\,sup}_{K}|u-h_{\eps_j}|
\leq1/j}$, we obtain
\[
\operatorname*{ess\,sup}_{K}|h_{\eps_k}-h_{\eps_l}|
\leq
\operatorname*{ess\,sup}_{K}|u-h_{\eps_k}|
+\operatorname*{ess\,sup}_{K}|u-h_{\eps_l}|
\leq\frac{1}{k}+\frac{1}{l}.
\]
Since $\{h_{\eps_k}\}$ is uniformly Cauchy on $K$, it converges uniformly on $K$ to a continuous function that 
equals $u$ a.e.\ in $K$. 
Because
$K\subset \Omega$ was an arbitrary closed ball, we can cover $\Omega$ with countably many open balls whose closures are contained in $\Omega$, and the argument above provides continuous representatives of $u$ on each ball.
 Any two of them agree on 
the overlaps, and hence they form a single continuous function equal to 
$u$ a.e.\ in $\Omega$, proving the claim.

Let us now show that $f$ agrees a.e.\ with a continuous function.
Replacing $u$ with its continuous representative, all the terms of the binomial expansion of 
$(A_\eps-I)^m[u]$ are continuous 
in ${\{x\in\Omega:\ \textnormal{dist}(x,\partial\Omega)>m\eps\}}$, and so is the function
\[
g_\eps
=\frac{\big(2(n+2)\big)^m}{\eps^{2m}}\,(A_\eps-I)^m[u].
\]
Let $K\subset\Omega$ be a closed ball and let $\eta'>0$. Applying 
\eqref{rho.smallness} with $\eta=\eta'\,(2(n+2))^{-m}$, and choosing any 
$\eps<\min\big\{\eps_0,\textnormal{dist}(K,\partial\Omega)/m\big\}$, we 
obtain from \eqref{AMVP.f.remainder} that
\[
\operatorname*{ess\,sup}_{K}|f-g_{\eps}|
=\frac{\big(2(n+2)\big)^m}{\eps^{2m}}
\operatorname*{ess\,sup}_{K}|\rho_{\eps}|
\leq\eta'.
\]
Arguing as before, we get that $f$ agrees a.e.\ in $\Omega$ with a continuous function, as desired.
\end{proof}

\subsection{Regularity and converse to the mean-value formula}

Let us show that a locally integrable function satisfying the mean-value formulas in \eqref{MVP_1} or \eqref{MVP_2} shares the same regularity as solutions to the higher-order Poisson equation $\Delta^mu=f$. This also provides a strong converse to the mean-value formulas, in the sense that merely locally integrable functions satisfying either of these are solutions of $\Delta^mu=f$ in an appropriate sense according to the regularity of $f$.

Let us compare these results to their counterparts in the homogeneous case $f\equiv0$ (see \cite{CLR}), where a strong converse holds, in the sense that $L^1_{loc}$ functions satisfying a mean-value property are smooth classical solutions of the polyharmonic equation $\Delta^m u=0$. 
The situation is different for nontrivial right-hand sides $f$. There is a rich regularity theory for the higher-order Poisson equation $\Delta^mu=f$ (see, for instance, \cite{Gazzola.et.al.2010}), where the different levels of regularity of $f$ translate into different levels of regularity of the solution $u$, which in turn dictates the appropriate notion~of~solution.

One should expect a similar effect for \eqref{MVP_1}--\eqref{MVP_2}, where the regularity that can be reached through the mean-value formula should depend on the regularity of $f$.
Likewise, strong converse results, where locally integrable functions satisfying the mean-value property are found to be solutions of the higher-order Poisson equation, are more subtle, in the sense that the notion of solution can also depend on the regularity of $f$.

\begin{proposition}[Regularity and converse to the mean-value property]\label{prop.regularity}
Let $\Omega\subset\R^n$, $n\geq2$, be open, $m$ a positive integer, and $f :\Omega\to\R$ with $f\in L^\infty_{loc}(\Omega)$. Assume there exists a real-valued $w \in C^{2m-1,1}(\Omega)$ with $\Delta^m w = f$ 
a.e.\ in $\Omega$. Let $u \in L^1_{\mathrm{loc}}(\Omega)$~satisfy
\begin{equation}\label{eq.MVPhyp}
(A_\varepsilon - I)^m[u](x)
= \int_{[0,\varepsilon]^m} A_{r}[f](x)\,
\prod_{i=1}^m K_\varepsilon(r_i)\, dr
\quad
\textrm{for a.e. } x \in \Omega \textrm{ and every } 
\varepsilon < \tfrac{\operatorname{dist}(x,\partial\Omega)}{m}.
\end{equation}
Then $u \in C^{2m-1,1}(\Omega)$ and $\Delta^m u = f$ a.e.\ in $\Omega$. 
Moreover, depending on $f$, the function $w$ may be more regular, in which case $u$ enjoys the same level of regularity as $w$.
\end{proposition}

\begin{remark}
The hypothesis \eqref{eq.MVPhyp} can equivalently be replaced by
\eqref{MVP_2}, since the right-hand sides of \eqref{MVP_1} and
\eqref{MVP_2} coincide for a.e.\ $x\in\Omega$; see
\eqref{B_integral.tilde.c's} below.
\end{remark}

\begin{proof}
By Theorem \ref{general.main}, the function $w$ satisfies 
\eqref{eq.MVPhyp} for every $x \in \Omega$. Hence, by linearity of $(A_\varepsilon - I)^m$, the difference
$v = u - w \in L^1_{\mathrm{loc}}(\Omega)$ satisfies
\[
(A_\varepsilon - I)^m[v](x) = 0
\quad \textrm{for a.e. } x \in \Omega \textrm{ and every } 
\varepsilon < \tfrac{\operatorname{dist}(x,\partial\Omega)}{m}.
\]
By Theorems \ref{thm.MVP.polyharmonic.intro} and \ref{thm.regularity} (see also \cite{CLR}), $v \in C^\infty(\Omega)$ and $\Delta^mv=0$, which implies
$u = w + v \in C^{2m-1,1}(\Omega)$ and
$\Delta^m u = f$ a.e.\ in $\Omega$, as desired. Observe that the relation $u = w + v$ with $v$ smooth means $u$ is as regular as $w$. 
\end{proof}

\begin{remark}
From the proof, it follows that
any two functions in $L^1_{\mathrm{loc}}(\Omega)$ satisfying 
\eqref{eq.MVPhyp} differ by an $m$-harmonic function.
\end{remark}

Let us finish this section discussing the role of the exact remainder in \eqref{MVP_1} or \eqref{MVP_2} in the proof of Proposition \ref{prop.regularity}. 
Whenever $f$ is continuous,
the same result can be obtained from an asymptotic mean-value property. More precisely, given $f\in C(\Omega)$, assume $u \in L^1_{\mathrm{loc}}(\Omega)$ satisfies~\eqref{AMVP.f}. Note that by Proposition \ref{prop.necessity.continuity}, $u \in C(\Omega)$. Let $w \in C^{2m-1,1}(\Omega)$ with $\Delta^m w = f$ a.e.\ in $\Omega$ as in Proposition~\ref{prop.regularity}, which,
by Theorem~\ref{general.main}, also satisfies \eqref{AMVP.f}.
Hence, the difference $v = u - w \in C(\Omega)$~satisfies
\[
(A_\eps-I)^m[v](x)
=
o(\eps^{2m})
\qquad\textrm{as}\ \eps\to0,\ \textrm{locally uniformly in}\ \Omega.
\]
Therefore, by Theorem \ref{thm.regularity}, $v \in C^\infty(\Omega)$, and we conclude that
$u$ is as regular as $w$.

The continuity of $f$ is a necessary and sufficient condition for the validity of the asymptotic expansion \eqref{AMVP.f}, as shown by Theorem \ref{general.main} and Proposition \ref{prop.necessity.continuity}. Thus, if one 
wishes to allow discontinuous forcing terms $f\in L^\infty_{loc}(\Omega)$, 
asymptotic mean-value properties are unavailable, and the exact formulas 
\eqref{MVP_1}--\eqref{MVP_2}, with the remainder quantified in terms of the 
oscillation of $f$, are the appropriate formulation.

We proceed with the proofs of Theorems \ref{MVP.pointwise.Laplacian} and \ref{general.main} and  Corollaries \ref{cor.spherical.m=1} and \ref{cor.spherical.higher}.

\section{Proof of Theorem \ref{MVP.pointwise.Laplacian} and Corollary \ref{cor.spherical.m=1}}
\label{section.m.is.1}

Before proceeding with the proof of Theorem \ref{MVP.pointwise.Laplacian}, we prove two auxiliary lemmas. The first one describes an important auxiliary function.

\begin{lemma}\label{lemma.w.delta.epsilon}
Let $0< \delta<\varepsilon$ and $\chi_r(x)=|B_r(0)|^{-1}\cdot\chi_{B_r(0)}(x)$. Then, the 
 equation 
\begin{equation}\label{problem.w}
 \Delta w = \chi_\varepsilon - \chi_\delta\quad\textrm{in}\ \mathbb{R}^n
\end{equation}
admits 
a global $C^{1,1}$ solution, compactly supported in $B_\varepsilon(0)$, given by
\begin{equation}\label{problem.w.sol.n}
w_{\delta,\varepsilon}(x)
=
\left\{
\begin{aligned}
 &\frac{1}{2\varepsilon^n|B_1(0)|}\left[\left(1-\frac{\varepsilon^n}{\delta^n}\right)\frac{|x|^2}{n}-\frac{\varepsilon^{2}}{n-2}\left(1-\frac{\varepsilon^{n-2}}{\delta^{n-2}}\right)\right] & &\textrm{if}\ 0\leq|x|<\delta&\\
 &\frac{1}{2\varepsilon^n|B_1(0)|}\left[\frac{|x|^2}{n}+\frac{2\varepsilon^n}{n(n-2)}|x|^{2-n}-\frac{\varepsilon^{2}}{n-2}\right] & &\textrm{if}\ \delta\leq|x|<\varepsilon,\\
 &0 & &\textrm{otherwise,}\\
\end{aligned}
\right.
\end{equation}
 when $n\geq3$,
and by
\begin{equation}\label{problem.w.sol.n=2}
w_{\delta,\varepsilon}(x)
=
\left\{
\begin{aligned}
&\frac{1}{2\varepsilon^2\pi}\left[\left(1-\frac{\varepsilon^2}{\delta^2}\right)
\frac{|x|^2}{2}+\varepsilon^2(\log{\varepsilon}-\log{\delta}) \right] 
 & &\textrm{if}\ 0\leq|x|<\delta\\
 &\frac{1}{2\varepsilon^2\pi}\left[
 \frac{|x|^2-\varepsilon^2}{2}+\varepsilon^2(\log{\varepsilon}-\log{|x|})
 \right] & &\textrm{if}\ \delta\leq|x|<\varepsilon,\\
 &0 & &\textrm{otherwise,}\\
\end{aligned}
\right.
\end{equation}
when $n=2$.

%

\end{lemma}

\begin{proof}
We can write \eqref{problem.w.sol.n} and \eqref{problem.w.sol.n=2} as 
\[
w_{\delta,\varepsilon}(x)
=
\left\{
\begin{aligned}
 &\phi(|x|) & &\textrm{if}\ 0\leq|x|<\delta\\
 &\psi(|x|) & &\textrm{if}\ \delta\leq|x|\leq\varepsilon\\
 &0 & &\textrm{otherwise.}\\
\end{aligned}
\right.
\]
Then, it is easy to check that \eqref{problem.w.sol.n} and \eqref{problem.w.sol.n=2} satisfy
\[
\left\{
\begin{aligned}
 &\frac{1}{r^{n-1}}\frac{d}{dr}\left(r^{n-1}\phi'(r)\right)=\frac{1}{\varepsilon^n|B_1(0)|}\left(1-\frac{\varepsilon^n}{\delta^n}\right)&\textrm{if}\ &0\leq r<\delta\\
 &\frac{1}{r^{n-1}}\frac{d}{dr}\left(r^{n-1}\psi'(r)\right)=\frac{1}{\varepsilon^n|B_1(0)|}\
 &\textrm{if}\ &\delta<r\leq\varepsilon\\
 &\phi(\delta)=\psi(\delta)\\
 &\phi'(\delta)=\psi'(\delta)\\
 &\psi(\varepsilon)=\psi'(\varepsilon)=0,\\
 \end{aligned}
 \right.
\]
for the corresponding $n$. Therefore $w_{\delta,\varepsilon}$ is a $C^{1,1}$ solution of \eqref{problem.w}, compactly supported in $B_\varepsilon(0)$.
\end{proof}

Next, we compute the limit as $\delta \to 0$ of the convolution of a locally integrable function with the function $w_{\delta,\varepsilon}$
 from the previous lemma.

\begin{lemma}\label{lemma.limit.of.convolution}
Let $0< \delta<\varepsilon$, $n\geq2$, and $w_{\delta,\varepsilon}$ as in Lemma \ref{lemma.w.delta.epsilon}. If $f\in L_\textrm{loc}^1(\mathbb{R}^n)$, then, for every $x$ a Lebesgue point of $f$, we have
\begin{equation}\label{formula.remainder}
\lim_{\delta\to0}\, (f\ast w_{\delta,\varepsilon})(x)
=
\frac{1}{n}
\int_{0}^{\varepsilon}
\dashint_{B_{r}(x)}
f(y)\,dy
\,
\left(r-\frac{r^{n+1}}{\varepsilon^n}\right)\,dr
=
\int_0^\eps A_r[f](x)\, K_{\varepsilon}(r)\,dr.
\end{equation}
Moreover, if $f\in L^{\infty}_{\textrm{loc}}(\mathbb{R}^n)$, then \eqref{formula.remainder} holds pointwise for all $x$.
\end{lemma}

\begin{proof}
Let us write \eqref{problem.w.sol.n} and \eqref{problem.w.sol.n=2} as 
\[
w_{\delta,\varepsilon}(x)
=
\left\{
\begin{aligned}
 &\phi(|x|) & &\textrm{if}\ 0\leq|x|<\delta\\
 &\psi(|x|) & &\textrm{if}\ \delta\leq|x|\leq\varepsilon\\
 &0 & &\textrm{otherwise.}\\
\end{aligned}
\right.
\]
Then, a change to polar coordinates gives
 \begin{equation}\label{convolution.polar.two.integrals}
 \begin{split}
(f\ast w_{\delta,\varepsilon})(x)
&=
\int_{B_\delta(0)}
f(x-y)\, \phi(|y|)\,dy
+
\int_{B_\varepsilon(0)\setminus B_\delta(0)}
f(x-y)\, \psi(|y|)\,dy
\\
&=
\int_{0}^{\delta}
F(r) \phi(r)r^{n-1}\,dr
+
\int_{\delta}^{\varepsilon}
F(r) \psi(r)r^{n-1}\,dr,
\end{split}
\end{equation}
 where we denote
 \[
 F(r)=\int_{\partial B_1(0)}
f(x-ry)\,d\mathcal{H}^{n-1}(y).
 \]
 Let us study both integrals on the right-hand side of \eqref{convolution.polar.two.integrals}. An integration by parts yields
\begin{equation}\label{estima.first part}
\begin{split}
\int_{0}^{\delta}
F(r) \phi(r)r^{n-1}\,dr
&=
\phi(\delta)
\int_{0}^{\delta}
F(s) s^{n-1}\,ds
-
\int_{0}^{\delta}
\int_{0}^{r}
F(s) s^{n-1}\,ds\;
\phi'(r)\,dr
\\
&=
|B_1(0)|\left(
\phi(\delta)\,\delta^n
\dashint_{ B_\delta(x)}
f(y)\,dy
-
\int_{0}^{\delta}
\dashint_{ B_r(x)}
f(y)\,dy
\;
\phi'(r)\,r^{n}\,dr
\right).
\end{split}
\end{equation}
We have that
\[
\phi(r)
=
\left\{
\begin{aligned}
&\frac{1}{2\varepsilon^n|B_1(0)|}\left[\left(1-\frac{\varepsilon^n}{\delta^n}\right)\frac{r^2}{n}-\frac{\varepsilon^{2}}{n-2}\left(1-\frac{\varepsilon^{n-2}}{\delta^{n-2}}\right)\right] 
 & &\textrm{if}\ n\geq3\\
&\frac{1}{2\varepsilon^2\pi}\left[\left(1-\frac{\varepsilon^2}{\delta^2}\right)
\frac{r^2}{2}+\varepsilon^2(\log{\varepsilon}-\log{\delta}) \right] 
 & &\textrm{if}\ n=2
\end{aligned}
\right.
\]
and therefore
\begin{equation}\label{estima1}
\phi(\delta)\,\delta^n\dashint_{ B_\delta(x)}
f(y)\,dy
=o(\delta)\qquad
\textrm{as}\ \delta\to0,
\end{equation} for every Lebesgue point $x$ of $f$ 
regardless of the dimension. Note also that it holds for all $x$ when $f$ is locally bounded. \par
On the other hand, for any dimension $n\geq2$ we have that
\[
\phi'(r)=\frac{1}{\varepsilon^n|B_1(0)|}\left(1-\frac{\varepsilon^n}{\delta^n}\right)\frac{r}{n}, 
\]
 and we can estimate
\begin{equation}\label{estima2}
\begin{split}
&\left|
\int_{0}^{\delta}
\dashint_{ B_r(x)}
f(y)\,dy
\;
\phi'(r)\,r^{n}\,dr
\right|
\\
&\hspace{52pt}\leq
\int_{0}^{\delta}
\dashint_{ B_r(x)}
|f(y)|\,dy
\;|\phi'(r)|\,r^{n}\,dr
\leq
\frac{C}{\delta^n}\int_{0}^{\delta}
r^{n+1}\,dr
=
C
\delta^{2}
\end{split}
\end{equation}
as $\delta\to0$. We conclude from \eqref{estima.first part}, \eqref{estima1}, and \eqref{estima2} that
\begin{equation}\label{first.integral.goes.to.zero}
\int_{0}^{\delta}
F(r) \phi(r)r^{n-1}\,dr\to0\qquad\textrm{as}\ \delta\to0.
\end{equation}

Similarly, an integration by parts (recall that $\psi(\varepsilon)=0$) yields
\begin{equation}\label{estima.second.part}
\begin{split}
\int_{\delta}^{\varepsilon}
F(r) \psi(r)r^{n-1}\,dr
&=
-
\int_{\delta}^{\varepsilon}
\int_{\delta}^{r}
F(s) s^{n-1}\,ds\,
\psi'(r)\,dr
\\
&=
-
\int_{\delta}^{\varepsilon}
\int_{B_{r}(x)\setminus B_\delta(x)}
f(y)\,dy
\,
\psi'(r)\,dr
\\
&=
-
\psi(\delta)
\int_{B_{\delta}(x)}
f(y)\,dy
-
\int_{\delta}^{\varepsilon}
\int_{B_{r}(x)}
f(y)\,dy
\,
\psi'(r)\,dr.
\end{split}
\end{equation} 
Regardless of the dimension $n$, we have that
\begin{equation}\label{estima2.1}
\psi(\delta)
\int_{B_{\delta}(x)}
f(y)\,dy
=
\psi(\delta)\delta^n|B_1(0)|
\dashint_{B_{\delta}(x)}
f(y)\,dy
=o(\delta)\qquad
\textrm{as}\ \delta\to0.
\end{equation}
On the other hand, for any dimension $n\geq2$ we have that
\begin{equation}\label{formula.psi'}
\psi'(r)r^n=\frac{r^{n+1}-r\varepsilon^n}{n\varepsilon^n|B_1(0)|},
\end{equation}
which is integrable near 0, and we get
\begin{equation}\label{estima2.2}
\int_{\delta}^{\varepsilon}
\dashint_{ B_r(x)}
f(y)\,dy
\;
\psi'(r)\,r^{n}\,dr
\to
\int_{0}^{\varepsilon}
\dashint_{ B_r(x)}
f(y)\,dy
\;
\psi'(r)\,r^{n}\,dr
\end{equation}
as $\delta\to0$. We conclude from \eqref{estima.second.part}, \eqref{estima2.1}, and \eqref{estima2.2} that
\begin{equation}\label{second.integral.goes.to.stuff}
\int_{\delta}^{\varepsilon}
F(r) \psi(r)r^{n-1}\,dr
\to
-
\int_{0}^{\varepsilon}
\int_{B_{r}(x)}
f(y)\,dy
\,
\psi'(r)\,dr
\end{equation}
as $\delta\to0.$

Finally, \eqref{formula.remainder} follows from \eqref{convolution.polar.two.integrals}, \eqref{first.integral.goes.to.zero}, and \eqref{second.integral.goes.to.stuff}.
\end{proof}

We are ready for the proof of Theorem \ref{MVP.pointwise.Laplacian}.
\begin{proof}[Proof of Theorem \ref{MVP.pointwise.Laplacian}]
Let $0< \delta<\varepsilon$ and $w_{\delta,\varepsilon}$ as in Lemma \ref{lemma.w.delta.epsilon}. Then,
 \[
 f\ast w_{\delta,\varepsilon}
=\Delta v \ast w_{\delta,\varepsilon}
=v \ast \Delta w_{\delta,\varepsilon}
=v \ast (\chi_\varepsilon - \chi_\delta),
\]
where $\chi_r(x)=|B_r(0)|^{-1}\cdot\chi_{B_r(0)}(x)$. Hence, by Lemma \ref{lemma.limit.of.convolution} and the continuity of $v$, we have
 \[
 \frac{1}{n}
\int_{0}^{\varepsilon}
\dashint_{B_{r}(x)}
f(y)\,dy
\,
\left(r-\frac{r^{n+1}}{\varepsilon^n}\right)\,dr
=
 \lim_{\delta\to0}\, (f\ast w_{\delta,\varepsilon})(x)
=
\dashint_{B_\varepsilon(x)} v(y)\,dy-v(x),
\]
which can be rewritten as \eqref{with_A}. 
 Then, noting the value of the integral of $K_{\varepsilon}(r)$, we can rewrite
\[
\begin{split}
 \frac{1}{n}
 &
\int_{0}^{\varepsilon}
\dashint_{B_{r}(x)}
f(y)\,dy
\,
\left(r-\frac{r^{n+1}}{\varepsilon^n}\right)\,dr
\\
&=
 \frac{1}{n}
\int_{0}^{\varepsilon}
\dashint_{B_{r}(x)}
\big(f(y)-f(x)\big)\,dy
\,
\left(r-\frac{r^{n+1}}{\varepsilon^n}\right)\,dr
+\frac{f(x)}{2(n+2)}\varepsilon^2
\end{split}
\]
for almost every $x\in\Omega$
and \eqref{MVP_lap} follows.

Assume now that $f$ is continuous at $x$. Then, for every $\eta>0$, there exists a $\tau>0$ such that $|x-y|<\tau$ yields 
$|f(x)-f(y)|<\eta$. In particular, for $\varepsilon<\tau$, we have that
\[
|R_\varepsilon(x)|\leq
 \frac{
\eta}{n}
\int_{0}^{\varepsilon}
\left(r-\frac{r^{n+1}}{\varepsilon^n}\right)\,dr= \frac{
\eta}{2(n+2)}\varepsilon^2.
\]
Since $\eta$ is arbitrary, we have that $R_\varepsilon(x)=o(\varepsilon^2)$.
If $f \in C(\Omega)$, then $f$
 is uniformly continuous on compact subsets and $R_\varepsilon(x)=o(\varepsilon^2)$ locally uniformly.
 
In the case that $f$ has a modulus of continuity $\omega$, we have
$$
\Big|A_r \big[f(x)-f(\cdot)\big](x)\Big|
\leq
\dashint_{B_{r}(x)}\omega\big(|x-y|\big)\,dy
\leq \omega(r)
$$
and, therefore,
\[
\begin{split}
|R_\varepsilon(x)|
&\leq
\int_{0}^{\varepsilon}
\omega(r)
K_{\varepsilon}(r)\,dr\leq
 \frac{\omega(\varepsilon)}{n}
\int_{0}^{\varepsilon}
\left(r-\frac{r^{n+1}}{\varepsilon^n}\right)\,dr=
\frac{1}{2(n+2)}\,\omega(\varepsilon)\,\varepsilon^2,
\end{split}
\]
as desired.
\end{proof}

Let us finish this section with the spherical-average counterpart of Theorem \ref{MVP.pointwise.Laplacian}.
\begin{proof}[Proof of Corollary \ref{cor.spherical.m=1}]
We define the differential operator $L_r=\frac{r}{n}\frac{\partial}{\partial r}+1$. 
For a continuous $u$, and \mbox{$0<r<\textnormal{dist}(x,\partial\Omega)$,} it is proved in  \cite{Lysik.2011} that
\[
L_r\Big(A_r[u](x)\Big)=\mathcal{A}_r[u](x),
\]
which can be obtained differentiating in $r$ the polar coordinate decomposition
\begin{equation}\label{coarea.identity}
A_r[u](x)=\frac{n}{r^n}\int_0^r t^{\,n-1}\mathcal{A}_t[u](x)\,dt.
\end{equation}
Thus, applying the operator $L_\eps$ to the left-hand side of \eqref{with_A}, we obtain the left-hand side of \eqref{MVP.spherical}.
For the right-hand side of \eqref{with_A}, we use  Leibniz's rule and the fact that $\frac{\partial}{\partial\eps}K_\eps(r)
=\frac{r^{n+1}}{\eps^{n+1}}$ and $K_\eps(\eps)=0$, which gives
\[
\frac{\partial}{\partial\eps}\int_0^\eps A_r[f](x)\,K_\eps(r)\,dr
=\int_0^\eps A_r[f](x)\,\frac{r^{n+1}}{\eps^{n+1}}\,dr,
\]
and therefore
\[
L_\eps\int_0^\eps A_r[f](x)\,K_\eps(r)\,dr
=\int_0^\eps A_r[f](x)
\left(\frac{r^{n+1}}{n\,\eps^{n}}
+\frac{r}{n}-\frac{r^{n+1}}{n\,\eps^{n}}\right)dr
=\int_0^\eps A_r[f](x)\,\frac{r}{n}\,dr,
\]
which is the first form of the right-hand side in \eqref{MVP.spherical}. Using \eqref{coarea.identity}, and exchanging the order of integration,
 we get the second form of the right-hand side in \eqref{MVP.spherical} in terms of the kernel \eqref{ K.spherical.def}.
The rest of the statements follow as in the proof of Theorem \ref{MVP.pointwise.Laplacian}.
\end{proof}

\section{Proof of Theorem \ref{general.main} and Corollary \ref{cor.spherical.higher}}\label{section.m.greater.1}

Before proceeding with the proof of Theorem \ref{general.main},
let us first derive some useful lemmas. We start by showing that the averaging operators $A_{r}$ commute with each other, and with integrals in the parameter $r$.

\begin{lemma}\label{commute}
Let $f\in L^1_{loc}(\Omega)$, $\rho,\tau,\eps>0$, and $w\in L^\infty(0,\eps)$.
\begin{enumerate}
\item For every $x\in\Omega$ with 
$\textnormal{dist}(x,\partial\Omega)>\rho+\tau$,
\begin{equation*}
A_\rho\big[A_\tau f\big](x)
=
A_\tau\big[A_\rho f\big](x),
\end{equation*}
all quantities being finite.
\item For every $x\in\Omega$ with 
$\textnormal{dist}(x,\partial\Omega)>\rho+\eps$, the function
$y\mapsto\int_0^\eps A_s[f](y)\,w(s)\,ds$ is defined for a.e.\ 
$y\in B_\rho(x)$, belongs to $L^1(B_\rho(x))$, and satisfies
\begin{equation*}
A_\rho\Big[\int_0^\eps A_s[f](\cdot)\,w(s)\,ds\Big](x)
=\int_0^\eps A_\rho\big[A_s f\big](x)\,w(s)\,ds.
\end{equation*}

\end{enumerate}
\end{lemma}

\begin{proof}
\noindent{}1.\quad{}Writing out the iterated average 
\[
A_\tau [A_\rho f](x) =\dashint_{B_\tau(x)}\dashint_{B_\rho (y)} f(z)\,dzdy,
\]
and applying Tonelli's theorem 
to $|f|$ (the integrand is supported in 
$\overline{B_{\rho+\tau}(x)}\subset\Omega$, where $f$ is integrable), 
Fubini's theorem applies. Using change of variables and exchanging the order of integration when necessary, we get 
\[
\begin{split}
A_\tau [A_\rho f]&(x) = \frac{1}{|B_1(0)|^2 \tau^n\rho^n}\int_{|y-x|\le \tau}\int_{|z-y|\le\rho} f(z)\,dzdy\\
&
=\frac{1}{|B_1(0)|^2 }\int_{|v|\le 1}\int_{|u|\le 1} f(\rho u + \tau v + x)\,dudv
=\frac{1}{|B_1(0)|^2 }\int_{|u|\le 1}\int_{|v|\le 1} f(\rho u + \tau v + x)\,dvdu\\
&
=\frac{1}{|B_1(0)|^2 \tau^n\rho^n}\int_{|z-x|\le\rho}\int_{|y-z|\le \tau} f(y)\,dydz
= A_\rho [A_\tau f](x).
\end{split}
\]

\noindent{}2.\quad{}The function $(y,s)\mapsto A_s[f](y)\,w(s)$ is jointly measurable 
on $B_\rho(x)\times(0,\eps)$. 
By~Tonelli theorem,
\[
\begin{split}
\dashint_{B_\rho(x)}A_s\big[|f|\big](y)\,dy
&=\frac{1}{|B_1(0)|^2\rho^n s^n}
\int_{B_\rho(x)}\int_{\R^n}\chi_{\{|z-y|\leq s\}}\,|f(z)|\,dzdy\\
&=\frac{1}{|B_1(0)|^2\rho^n s^n}
\int_{\R^n}|f(z)|\int_{B_\rho(x)}\chi_{\{|y-z|\leq s\}}\,dydz\\
&=\frac{1}{|B_1(0)|^2\rho^n s^n}
\int_{B_{\rho+s}(x)}|f(z)|\cdot\big|B_\rho(x)\cap B_s(z)\big|\,dz\\
&\leq\frac{1}{|B_1(0)|\rho^n}\,
\int_{B_{\rho+\eps}(x)}|f(z)|\,dz
=\frac{\|f\|_{L^1(B_{\rho+\eps}(x))}}{|B_1(0)|\,\rho^n},
\end{split}
\]
where in the third line we have used that
$\big|B_\rho(x)\cap B_s(z)\big|=0$ when 
$z\notin B_{\rho+s}(x)$. Therefore,
\[
\int_0^\eps \dashint_{B_\rho(x)} A_s\big[|f|\big](y)\,dy\;|w(s)|\,ds
\leq \frac{\eps}{|B_1(0)|\,\rho^n}\,
\|w\|_{L^\infty(0,\eps)}\,\|f\|_{L^1(B_{\rho+\eps}(x))}<\infty.
\]
Fubini's theorem now yields all the assertions.
\end{proof}

We will also need the following
elementary estimates for iterated averages.

\begin{lemma}\label{lemma.osc}
Let $f\in L^\infty_{loc}(\Omega)$ and $x\in\Omega$.
\begin{enumerate}\itemsep3pt
\item If $r_1,\dots,r_j>0$ with
$r_1+\cdots+r_j<\textnormal{dist}(x,\partial\Omega)$, then
\[
\big|\big(A_{r_1}\!\cdots A_{r_j}\big)[f](x)-f(x)\big|
\leq
\sup\Big\{|f(z)-f(x)|:\ z\in B_{r_1+\cdots+r_j}(x)\Big\}.
\]
\item If $l\in\mathbb{N}$ and
$l\eps<\textnormal{dist}(x,\partial\Omega)$, then
\[
\big|(A_\eps-I)^l[f](x)\big|
\leq
\big(2^{l}-1\big)\,
\sup\Big\{|f(z)-f(x)|:\ z\in B_{l\eps}(x)\Big\}.
\]
\end{enumerate}
In particular, if $f$ has modulus of continuity $\omega$ in $\Omega$, the two
right-hand sides are bounded by $\omega(r_1+\cdots+r_j)$ and
$(2^l-1)\,\omega(l\eps)$, respectively.
\end{lemma}

\begin{proof}
\noindent{}1.\quad{}For any $g\in L^\infty_{loc}$, and any 
$B_s(x)\subset\Omega$, we have $|A_s[g](x)|\le\sup_{B_s(x)}|g|$.
The claim follows applying this estimate iteratively, and using that every value of $g(y)=f(y)-f(x)$
sampled by $A_{r_1}\!\cdots A_{r_j}$ is taken at points in $B_{r_1+\cdots+r_j}(x)$.

\noindent{}2.\quad{}Expanding the binomial and using that its coefficients sum to zero, we get
\[
(A_\eps-I)^l
=\sum_{j=0}^{l}\binom{l}{j}(-1)^{l-j}A_\eps^{j}
=\sum_{j=1}^{l}\binom{l}{j}(-1)^{l-j}\big(A_\eps^{j}-I\big).
\]
By part 1 with $r_1=\cdots=r_j=\eps$, each term satisfies
\[
\big|\big(A_\eps^{j}-I\big)[f](x)\big|\le
\sup_{y\in B_{j\eps}(x)}|f(y)-f(x)|\le\sup_{y\in B_{l\eps}(x)}|f(y)-f(x)|
\]
 for $j\le l$. Using that 
$\sum_{j=1}^{l}\binom{l}{j}=2^l-1$, we get the result.
\end{proof}

In the next lemma we record a useful relation between the Pizzetti coefficients in \eqref{pizz_c} and certain iterated weighted integrals of the kernel $K_\varepsilon$.

\begin{lemma}\label{Kint}
Let $K_\eps$ be as in \eqref{K_def}. Then,
for every integer $j\geq1$, we have
\begin{equation}\label{integral.K_eps}
\int_0^\eps 
\left(\frac{r}{\varepsilon}\right)^{2j-2}K_{\varepsilon}(r)\,dr =
\frac{\eps^2}{(2j)(n+2j)}.
\end{equation}
In particular, the Pizzetti coefficients $c_\alpha$ in \eqref{pizz_c} can be expressed as follows:
\begin{equation}\label{pizzetti.integral.form}
c_\alpha= 
\prod_{j=1}^\alpha\frac{1}{(2j)(n+2j)}
=\frac{1}{\eps^{2\alpha}}
\int_0^\eps\!\cdots \int_0^\eps\prod_{j=1}^\alpha
\left(\frac{r_j}{\varepsilon}\right)^{2j-2}K_{\varepsilon}(r_j)\,dr_1\ldots dr_\alpha.
\end{equation}
\end{lemma}

\begin{proof}
Identity \eqref{integral.K_eps} follows from direct integration; \eqref{pizzetti.integral.form} then follows from Fubini's theorem.
\end{proof}

\begin{remark}\label{remark.n.minus.2}
In line with Remark \ref{remark.K.mathcal.K.dimension.intro}, replacing $n$ by $n-2$ in
Lemma \ref{Kint} gives the moments of $\mathcal{K}_\eps$
when $n\geq3$, namely
\[
\int_0^\eps\Big(\frac{r}{\eps}\Big)^{2j-2}\mathcal{K}_\eps(r)\,dr
=\frac{\eps^2}{(2j)(n+2j-2)},
\vspace{+3pt}
\]
and the formula remains valid for $n=2$ by monotone convergence,
since the integrands increase to their limit as $n\to2^+$.
\end{remark}

Finally, let us prove some estimates and integrals involving the coefficients $b_{\eps}^{m,k}(r)$.

\begin{lemma}\label{lemma.B.properties}
Let $1\le k\le m$, $r=(r_1,\ldots, r_m)$, and $b_{\eps}^{m,k}(r)$ defined by \eqref{B_def}.
The following hold.
\begin{enumerate}\itemsep3pt
\item\emph{Kernel factorization:}
\begin{equation}\label{B_kernels}
\sum_{\alpha\in\mathcal{C}(m,k)}\;
\prod_{i=1}^{k}\prod_{j=1}^{\alpha_i}\left(
\left(\frac{r_{\sigma_{i-1}+j}}{\eps}\right)^{2j-2}K_\eps(r_{\sigma_{i-1}+j})
\right)
=
b_{\eps}^{m,k}(r)\cdot\prod_{q=1}^{m}K_\eps(r_q).
\end{equation}

\item \emph{Bounds:}
For every
$r\in[0,\eps]^m$, we have the bounds
\begin{equation}\label{item.bound}
 0\le b^{m,k}_\eps(r)\le\binom{m-1}{k-1}.
 \end{equation}
 
\item\emph{Integrals:}
\begin{equation}\label{B_integral}
\int_{[0,\eps]^m}b_{\eps}^{m,k}(r)\,\prod_{q=1}^{m}K_\eps(r_q)\,dr
=\eps^{2m}\sum_{\alpha\in\mathcal{C}(m,k)}\;\prod_{i=1}^{k}c_{\alpha_i}.
\end{equation}
In particular, for $0\le l\le m-1$, we have
\begin{equation}\label{B_integral.tilde.c's}
\int_{[0,\eps]^m}b_{\eps}^{m,m-l}(r)\,\prod_{q=1}^{m}K_\eps(r_q)\,dr
=
\eps^{2m}\,\tilde c_l
\end{equation}
 for $\tilde c_l$ as in \eqref{pizz_c}.
\end{enumerate}
\end{lemma}

\begin{proof}
\noindent{}1.\quad{}For each fixed term $\alpha$ in the composition,
we can use \eqref{block.bijection} to
re-index the product of kernels. Namely, as $(i,j)$ ranges over the index set, the global
index $\sigma_{i-1}+j$ ranges over $\{1,\dots,m\}$, each value attained
exactly once. We obtain
\[
\prod_{i=1}^{k}\prod_{j=1}^{\alpha_i}K_\eps\big(r_{\sigma_{i-1}+j}\big)
=
\prod_{q=1}^{m}K_\eps(r_q),
\]
 which factors out of the sum over $\mathcal{C}(m,k)$, and gives
\eqref{B_kernels}.

\noindent{}2.\quad{}Non-negativity is clear from the definition \eqref{B_def}. To get the upper bound, 
it suffices to count the number of terms in the sum,
since every factor in \eqref{B_def} is at most $1$. But then, the map
$\alpha\mapsto\{\sigma_1+1,\dots,\sigma_{k-1}+1\}$, which records the starting
positions of the blocks $I_2,\dots,I_k$, is a bijection from
$\mathcal{C}(m,k)$ onto the $(k-1)$-element subsets of $\{2,\dots,m\}$. Hence
$|\mathcal{C}(m,k)|=\binom{m-1}{k-1}$.

\noindent{}3.\quad{}Using the kernel factorization \eqref{B_kernels}, we can write the integral in terms of the variables $\{r_{\sigma_{i-1}+j}\}_{i,j}$. Since,
by the 
bijection \eqref{block.bijection}, 
each of the variables
$r_1,\dots,r_m$ appears as $r_{\sigma_{i-1}+j}$
for exactly one pair $(i,j)$, we can use Fubini's theorem to split the $m$-dimensional integral over $[0,\eps]^m$ into the product of $m$
one-dimensional integrals
\[
\begin{split}
\int_{[0,\eps]^m}b_{\eps}^{m,k}(r)\,\prod_{q=1}^{m}K_\eps(r_q)\,dr
&=
\sum_{\alpha\in\mathcal{C}(m,k)}
\int_{[0,\eps]^m}
\prod_{i=1}^{k}\prod_{j=1}^{\alpha_i}\left(
\left(\frac{r_{\sigma_{i-1}+j}}{\eps}\right)^{2j-2}K_\eps(r_{\sigma_{i-1}+j})
\right)
dr\\
&=
\sum_{\alpha\in\mathcal{C}(m,k)}
\prod_{i=1}^{k}
\;
\prod_{j=1}^{\alpha_i}
\;
\int_0^\eps
\left(\frac{r_{\sigma_{i-1}+j}}{\eps}\right)^{2j-2}K_\eps(r_{\sigma_{i-1}+j})
\;
dr_{\sigma_{i-1}+j}
\\
&=
\sum_{\alpha\in\mathcal{C}(m,k)}
\;
\prod_{i=1}^{k}
c_{\alpha_i}\,\eps^{2\alpha_i}
=
\eps^{2m}\,
\sum_{\alpha\in\mathcal{C}(m,k)}
\;
\prod_{i=1}^{k}
c_{\alpha_i},
\end{split}
\]
where we have used \eqref{integral.K_eps} in the third step. The proof of \eqref{B_integral} is complete, and
 taking
$k=m-l$, and comparing with \eqref{pizz_c} yields~\eqref{B_integral.tilde.c's}.
\end{proof}

For illustration purposes, let us present first the proof of Theorem \ref{general.main} in the biharmonic case.

\begin{proof}[Proof of Theorem \ref{general.main} in the biharmonic case]
If $m=2$, we can rewrite $\lap^2 u = f$ in the form 
 \[
 \lap u = v\qquad\textrm{with}\qquad
 \lap v = f.
\]
Then, since $u$ satisfies Poisson's equation with source term $v$, we can apply \eqref{with_A} to give
\begin{equation}\label{foo}
(A_\eps-I)[u](x)=
\int_{0}^{\varepsilon}
A_{r_1} [v](x)\cdot K_{\varepsilon}({r_1})\,dr_1.
\end{equation}
Next, we apply the operator $(A_\eps - I)$ to both sides of \eqref{foo}; recalling Lemma \ref{commute}, we obtain
\begin{align*}
(A_\eps-I)^2[u](x)&=
(A_\eps - I)\int_{0}^{\varepsilon}
A_{r_1} [v](x)\cdot K_{\varepsilon}({r_1})\,dr_1\nonumber\\
&=\int_0^\eps A_{r_1}(A_\eps - I)[v](x)\cdot K_{\varepsilon}({r_1})\,dr_1.
\end{align*}
 Since $v$ also satisfies Poisson's equation, but with source term $f$, we can apply again \eqref{with_A} and Lemma \ref{commute} to obtain
\begin{align*}
(A_\eps-I)^2[u](x) &= \int_0^\eps A_{r_1} \left( \int_0^\eps A_{r_2} [f](x)\cdot K_{\varepsilon}({r_2})\,dr_2\right) \cdot K_{\varepsilon}({r_1})\,dr_1\\
&= \int_0^\eps\int_0^\eps \big(A_{r_1}A_{r_2}\big)[f](x)\cdot K_{\varepsilon}({r_1}) K_{\varepsilon}({r_2})\,dr_1 dr_2,
\end{align*}
which gives \eqref{MVP_1} when $m=2$.

From Lemma \ref{Kint}, we get
\[
\begin{split}
\int_0^\eps \int_0^\eps &\big(A_{r_1}A_{r_2}\big)[f](x)\cdot K_{\varepsilon}({r_1}) K_{\varepsilon}(r_2)\,dr_1 dr_2\\
&=\int_0^\eps \int_0^\eps \left(A_{r_1}A_{r_2}-I- \frac{{r_2}^2}{\eps^2}(A_\eps -I)\right)[f](x)\cdot K_{\varepsilon}({r_1}) K_{\varepsilon}(r_2)\,dr_1 dr_2\\
&
\quad+f(x)\int_0^\eps \int_0^\eps K_{\varepsilon}({r_1}) K_{\varepsilon}(r_2)\,dr_1 dr_2
+\frac{(A_\eps -I)[f](x)}{\eps^2}\int_0^\eps \int_0^\eps {r_2}^2 K_{\varepsilon}({r_1}) K_{\varepsilon}(r_2)\,dr_1 dr_2 \\
&= R_{\eps}(x) + \eps^4\left(\frac{1}{4(n+2)^2}f(x)+ \frac{1}{8(n+4)(n+2)}(A_\eps -I)[f](x)\right).
\end{split}
\]
Note that the two constants in the last line are precisely $\tilde{c}_0=c_1^2=\frac{1}{4(n+2)^2}$ 
and $\tilde{c}_1=c_2=\frac{1}{8(n+2)(n+4)}$, corresponding to $\alpha_1=\alpha_2=1$ and $\alpha_1=2$ in \eqref{pizz_c},
in agreement with \eqref{MVP_2}.

Now, we can estimate $R_{\eps}(x)$ as
\begin{equation}\label{R_est1}
|R_{\eps}(x)|\le\int_0^\eps \int_0^\eps \left|\big(A_{r_1}A_{r_2}\big)[f](x)-f(x)- \frac{{r_2}^2}{\eps^2}(A_\eps -I)[f](x)\right|\cdot K_{\varepsilon}({r_1}) K_{\varepsilon}(r_2)\,dr_1 dr_2.
\end{equation}
If we assume that $f$ is continuous at $x$, then, for every $\eta>0$ there exists a $\delta>0$ such that $|f(z)-f(x)| < \eta$ whenever $|z-x|<\delta$. 
Let $\eps < \delta/2$. Then, for every $r_1, r_2 \in (0,\eps),$ every $y \in B_{r_1}(x)$
 and every $z \in B_{r_2}(y)$ it holds that $|z - x| \le r_1 + r_2 < 2\eps < \delta$
 and we have
\begin{equation}\label{m=2.modulus.est.2}
\begin{split}
&\left|\big(A_{r_1}A_{r_2}\big)[f](x)-f(x)- \frac{{r_2}^2}{\eps^2}(A_\eps -I)[f](x)\right|
\\
&\qquad\leq
\dashint_{B_{r_1}(x)} \!\dashint_{B_{r_2}(y)} \! \left|f(z)-f(x)- \frac{{r_2}^2}{\eps^2}(A_\eps -I)[f](x)\right|\,dzdy\\
&\qquad\leq
\dashint_{B_{r_1}(x)} \!\dashint_{B_{r_2}(y)} \! \left(|f(z)-f(x)| + \frac{{r_2}^2}{\eps^2} \dashint_{B_\eps(x)}|f(t)- f(x)|\,dt\right)\,dzdy
\leq\eta + \frac{{r_2}^2}{\eps^2}\eta
\leq2\eta.
\end{split}
\end{equation}
Combining \eqref{R_est1} and \eqref{m=2.modulus.est.2}, we obtain
 \[
 |R_{\eps}(x)|
 \le
 2\eta
 \int_0^\eps \int_0^\eps 
 K_{\varepsilon}({r_1}) K_{\varepsilon}(r_2)\,dr_1 dr_2
 = \frac{\eta}{2(n+2)^2}\,\eps^4,
\]
as desired. If $f \in C(\Omega)$, then $f$
 is uniformly continuous on compact subsets and $R_\varepsilon(x)=o(\varepsilon^4)$ locally uniformly.
 
In the case that $f$ has a modulus of continuity $\omega$, we can estimate as in \eqref{R_est1}--\eqref{m=2.modulus.est.2} to get
\[
\begin{split}
&\left|\big(A_{r_1}A_{r_2}\big)[f](x)-f(x)- \frac{{r_2}^2}{\eps^2}(A_\eps -I)[f](x)\right|
\\
&\qquad\qquad\leq
\dashint_{B_{r_1}(x)} \!\dashint_{B_{r_2}(y)} \!
\omega(|z-x|)\,dzdy+\frac{{r_2}^2}{\eps^2}\omega(\eps)
\leq
\omega(r_1+r_2)+\frac{{r_2}^2}{\eps^2}\omega(\eps)
\leq 2\,\omega(2\eps),
\end{split}
\]
and, therefore,
\[
|R_\varepsilon(x)|
\le
2\,\omega(2\eps)\int_0^\eps \int_0^\eps 
K_{\varepsilon}({r_1}) K_{\varepsilon}(r_2)\,dr_1 dr_2\leq
\frac{1}{2(n+2)^2}\,\omega(2\eps)\,\eps^4,
\]
which completes the proof in the biharmonic case.
\end{proof}

We are ready for the proof of Theorem \ref{general.main}.

\begin{proof}[Proof of Theorem \ref{general.main}]
We proceed in three steps: first we establish the mean-value property
\eqref{MVP_1} with remainder in integral form, then we derive \eqref{MVP_2} with the remainder $R_\eps$, and finally we prove the
estimates (1)--(3).

\noindent1.\quad{}Let us prove \eqref{MVP_1} by induction. The proof for the base case $m=1$ is given in Theorem \ref{MVP.pointwise.Laplacian}. In particular, equations \eqref{with_A}, \eqref{MVP_lap}, and \eqref{R_MVP_lap} correspond to \eqref{MVP_1}, \eqref{MVP_2}, and \eqref{R_MVP} when $m=1$.
For the inductive step, assume that the claim is true for $m=k-1$, i.e., assume that
whenever $\lap^{k-1} w = f,$ then
\begin{equation}\label{inductive}
(A_\eps-I)^{k-1}[w](x) 
= \int_{[0,\eps]^{k-1}} \big(A_{r_1}\!\cdots A_{r_{k-1}}\big)[f](x)\cdot\prod_{i=1}^{k-1}K_\eps(r_i)\,dr_1\ldots dr_{k-1}.
\end{equation}
Also, we have $\lap^k u = f$, which we can rewrite as
 \[
 \lap u = v\qquad\textrm{with}\qquad
 \lap^{k-1}v = f.
 \]
Then, since $u$ satisfies Poisson's equation with source term $v$, we can apply \eqref{with_A} to give
\begin{equation}\label{step1}
(A_\eps-I)[u](x)=
\int_{0}^{\varepsilon}
A_{r_k} [v](x)\cdot K_{\varepsilon}({r_k})\,dr_k,
\end{equation}
where we are using $r_k$ as the variable of integration for convenience.
Next, we apply the operator $(A_\eps - I)^{k-1}$ to both sides of \eqref{step1}. Applying Lemma \ref{commute} iteratively, we obtain
\begin{align*}
(A_\eps-I)^k[u](x)&=
(A_\eps - I)^{k-1}\int_{0}^{\varepsilon}
A_{r_k} [v](x)\cdot K_{\varepsilon}({r_k})\,dr_k\nonumber\\
&=\int_0^\eps A_{r_k}(A_\eps - I)^{k-1}[v](x)\cdot K_{\varepsilon}({r_k})\,dr_k.
\end{align*}
 Since $v$ satisfies $\lap^{k-1} v = f$, we use our inductive assumption \eqref{inductive}~to~get
\[
\begin{split}
(A_\eps-I)^k[u](x) &= \int_0^\eps A_{r_k} \left( \int_{[0,\eps]^{k-1}} \big(A_{r_1}\!\cdots A_{r_{k-1}}\big)[f](x)\cdot\prod_{i=1}^{k-1}K_\eps(r_i)\,dr_1\ldots dr_{k-1}
\right) \cdot K_{\varepsilon}({r_k})\,dr_k\\
&= \int_{[0,\eps]^{k}} \big(A_{r_1}\!\cdots A_{r_{k}}\big)[f](x)\cdot\prod_{i=1}^{k}K_\eps(r_i)\,dr_1\ldots dr_{k},
\end{split}
\]
which gives \eqref{MVP_1} when $m=k$, and, by induction, for all $m\in\mathbb{N}$.

\noindent2.\quad{}Let us now prove \eqref{MVP_2} with the remainder $R_\eps$ given by \eqref{R_MVP}, which amounts to showing that the difference
\begin{equation}\label{main.theorem.step.2.eq1}
(A_\eps-I)^m[u](x)-\eps^{2m}\sum_{l=0}^{m-1}\tilde c_l\,(A_\eps-I)^l[f](x),
\end{equation}
admits the integral representation \eqref{R_MVP}.
This is easily seen once we observe that the factors $(A_\eps-I)^l[f](x)$ do not depend on
$r$, which, by \eqref{B_integral.tilde.c's}, implies
\begin{equation}\label{main.theorem.step.2.eq2}
\eps^{2m}\,\tilde c_l\,(A_\eps-I)^l[f](x)
=
\int_{[0,\eps]^m}b_{\eps}^{m,m-l}(r)\,(A_\eps-I)^l[f](x)\cdot\prod_{q=1}^{m}K_\eps(r_q)\,dr.
\end{equation}
Then, \eqref{main.theorem.step.2.eq1}, \eqref{main.theorem.step.2.eq2}, and 
 \eqref{MVP_1} give \eqref{R_MVP}.

\noindent3.\quad{}We will now prove the estimates (1)--(3).
 The remainder representation \eqref{R_MVP} places the pointwise value $f(x)$ and the
correction terms under the same weighted integral as $A_r[f](x)$.
Since $b^{m,m}_\eps=1$, the integrand of
\eqref{R_MVP} can be split as
\[
\Big(\big(A_{r_1}\!\cdots A_{r_m}\big)[f](x)-f(x)\Big)
-\sum_{l=1}^{m-1}b^{m,m-l}_\eps(r)\,(A_\eps-I)^l[f](x).
\]
Therefore, using $K_\eps\geq0$, the bounds \eqref{item.bound}, and
the integral \eqref{integral.K_eps} applied with $j=1$ in each variable~$r_i$, we obtain
\begin{equation}\label{R.split}
\begin{split}
|R_\eps(x)|
\le\frac{\eps^{2m}}{\big(2(n+2)\big)^m}
\Bigg(
\sup_{r\in[0,\eps]^m}
\Big|\big(A_{r_1}\!\cdots A_{r_m}\big)[f](x)&-f(x)\Big|
\\
+&\sum_{l=1}^{m-1}\binom{m-1}{l}\,\big|(A_\eps-I)^l[f](x)\big|
\Bigg).
\end{split}
\end{equation}

Assume that $f$ is continuous at $x$, and let $\eta>0$. Choose
$\delta>0$ such that $|f(z)-f(x)|<\eta$ for all $z\in B_\delta(x)\cap\Omega$,
and let $\eps<\min\{\delta,\textnormal{dist}(x,\partial\Omega)\}/m$, so that every
value of $f$ sampled in \eqref{R.split} is taken at points of
$B_{m\eps}(x)\subset B_\delta(x)$. By Lemma \ref{lemma.osc}, 
\[
\sup_{r\in[0,\eps]^m}
\Big|\big(A_{r_1}\!\cdots A_{r_m}\big)[f](x)-f(x)\Big|\leq\eta
\]
and 
\[
\sum_{l=1}^{m-1}\binom{m-1}{l}\,\big|(A_\eps-I)^l[f](x)\big|
\leq
\sum_{l=1}^{m-1}\binom{m-1}{l}(2^l-1)\,\eta.
\]
Hence, using the binomial theorem to simplify constants, we have
\[
|R_\eps(x)|
\le\frac{\eta\,\eps^{2m}}{\big(2(n+2)\big)^m}
\left(1+\sum_{l=1}^{m-1}\binom{m-1}{l}\big(2^{l}-1\big)\right)
=\frac{3^{m-1}-2^{m-1}+1}{\big(2(n+2)\big)^m}\,\eta\,\eps^{2m},
\]
and, since $\eta>0$ is arbitrary, $R_\eps(x)=o(\eps^{2m})$ as $\eps\to0$. 
If $f \in C(\Omega)$, then $f$
 is uniformly continuous on compact subsets and $R_\varepsilon(x)=o(\varepsilon^{2m})$ locally uniformly.
 
 In the case that $f$ has a modulus of continuity $\omega$, by Lemma
\ref{lemma.osc}, we have
\[
\Big|\big(A_{r_1}\!\cdots A_{r_m}\big)[f](x)-f(x)\Big|\le\omega(r_1+\cdots+r_m)
\le\omega(m\eps)
\]
and
\[
\big|(A_\eps-I)^l[f](x)\big|\le(2^l-1)\,\omega(l\eps)
\le(2^l-1)\,\omega(m\eps),
\]
 so \eqref{R.split} yields
\[
|R_\eps(x)|
\le\frac{3^{m-1}-2^{m-1}+1}{\big(2(n+2)\big)^m}\,\omega(m\eps)\,\eps^{2m},
\]
which completes the proof.
\end{proof}

Let us conclude with the spherical-average counterpart of Theorem \ref{general.main}.

\begin{proof}[Proof of Corollary \ref{cor.spherical.higher}]
Let us start from \eqref{eq.multiradius}.
Since the radii in \eqref{eq.multiradius} are independent, the application of an operator
$L_{\eps_i}=\frac{\eps_i}{n}\frac{\partial}{\partial\eps_i}+1$ to the left-hand side as in the proof of Corollary \ref{cor.spherical.m=1},
replaces the factor $(A_{\eps_i}-I)$ by $(\mathcal{A}_{\eps_i}-I)$.
 For the right-hand
side of \eqref{eq.multiradius}, applying Leibniz's rule one $L_{\eps_i}$ at a time, as in the proof of the case $m=1$ and setting  $\eps_1=\ldots=\eps_m=\eps$ yields \eqref{MVP.spherical.higher}.
The second form of the right-hand side in \eqref{MVP.spherical.higher} follows inserting the polar
coordinate decomposition \eqref{coarea.identity} in each variable
$r_i$ and exchanging the order of integration.

For \eqref{MVP_2.spherical} and the remainder estimates, we follow Steps 2
and 3 of the proof of Theorem \ref{general.main}, with the second form of the
right-hand side of \eqref{MVP.spherical.higher} in place of \eqref{MVP_1}.
Only two ingredients need to be replaced. First, by Remark
\ref{remark.n.minus.2}, the moments of $\mathcal{K}_\eps$ are
$\eps^{2}/\big((2j)(n+2j-2)\big)$ for all $n\geq2$; therefore, since the
coefficients $b^{m,k}_\eps$ in \eqref{B_def} are purely combinatorial and do
not depend on the dimension, Lemma \ref{lemma.B.properties} holds with
$\mathcal{K}_\eps$ in place of $K_\eps$ and $d_\alpha$ in place of
$c_\alpha$; in particular,
\[
\int_{[0,\eps]^m}b^{m,m-l}_\eps(r)\,\prod_{i=1}^m\mathcal{K}_\eps(r_i)\,dr
=\eps^{2m}\,\tilde d_l.
\]
Subtracting $\sum_{l=0}^{m-1}b^{m,m-l}_\eps(r)\,(\mathcal{A}_\eps-I)^l[f](x)$ 
under the integral in \eqref{MVP.spherical.higher} then yields 
\eqref{MVP_2.spherical}, for a.e.\ $x$ and a.e.\ $\eps$ so that all the 
spherical means of $f$ involved are defined. Second, Lemma \ref{lemma.osc} 
holds for iterated spherical means with the same proof. The estimates (1)--(3) then follow as in 
Step 3, with appropriate changes to the constants.
\end{proof}

\end{document}